\documentclass[11pt,a4paper,twoside]{article}

\usepackage{ifthen} 
\newboolean{ShowLabels}
	\setboolean{ShowLabels}{false}    	 

\usepackage[utf8]{inputenc}
\usepackage[T1]{fontenc}
\usepackage{lipsum} 				
\usepackage{amsmath} 				
\usepackage{amssymb}				
\usepackage{cases} 					
\usepackage{amsthm} 				
\usepackage{amsfonts} 				
\usepackage{microtype} 				
\usepackage{lmodern}				
\usepackage{calrsfs} 				
\usepackage[dvips]{graphicx} 	 	
\usepackage{xcolor} 				
\usepackage{soul}					
\usepackage{pdfcomment} 			
\usepackage[nameinlink]{cleveref}	
\usepackage[a4paper,left=2.5cm,right=2.5cm,top=3cm,bottom=3cm]{geometry}	
\usepackage{lmodern} 				
\usepackage{enumitem} 				
\usepackage{titletoc} 				
\usepackage{fancyhdr} 				
\usepackage{datetime} 				

	\definecolor{blue}{rgb}{0,0,.6}
	\definecolor{green}{rgb}{0,.4,0}

	\hypersetup{
		bookmarksopen=true,
		colorlinks=true, 
		urlcolor=black,  
		linkcolor=blue,  
		citecolor=blue   
	}

		\fancypagestyle{TitleStyle}
		{
		\fancyhf{}
		
		\fancyfoot[C]{
				\thepage
			}
		}

		\fancypagestyle{RegularStyle}{
		\fancyhf{} 
		\fancyhead[CE]{
				\footnotesize Samuel \textsc{Tréton} and Mingmin {\sc Zhang}
			}
		\fancyhead[CO]{
				\footnotesize The influence of an inward-shifting boundary on the heat equation in half-space
			}
		\fancyfoot[C]{
				\thepage
			}
		}

	\titlecontents{section}
		[6mm]
		{}
		{\contentslabel{6mm}}
		{}
		{\dotfill\contentspage}

	\ifthenelse{\boolean{ShowLabels}}{
		\usepackage[notcite,notref]{showkeys}
		
		}{
		}

		\DeclareMathAlphabet{\pazocal}{OMS}{zplm}{m}{n}
		\renewcommand{\mathcal}[1]{\pazocal{#1}}
		
		\newcommand{\Lb}{\pazocal{L}}

		\newcommand{\Sb}{\pazocal{S}}
		\newcommand{\md}{\mathrm{d}}

		\newtheoremstyle{mystyle}
		{2ex} 
		{2ex} 
		{\itshape} 
		{} 
		{\bfseries} 
		{} 
		{1em} 
		{} 

	\theoremstyle{mystyle}

		\newtheorem{theorem}{Theorem}[section]
		\newtheorem{lemma}[theorem]{Lemma}
		\newtheorem{proposition}[theorem]{Proposition}

		\newtheorem{remark}[theorem]{Remark}

	\crefname{theorem}{theorem}{theorems}
	\Crefname{theorem}{Theorem}{Theorems}

	\crefname{lemma}{lemma}{lemmas}
	\Crefname{lemma}{Lemma}{Lemmas}

	\crefname{proposition}{proposition}{propositions}
	\Crefname{proposition}{Proposition}{Propositions}

	\crefname{corollary}{corollary}{corollaries}
	\Crefname{corollary}{Corollary}{Corollaries}

	\crefname{conjecture}{conjecture}{conjectures}
	\Crefname{conjecture}{Conjecture}{Conjectures}

	\crefname{definition}{definition}{definitions}
	\Crefname{definition}{Definition}{Definitions}

	\crefname{assumption}{assumption}{assumptions}
	\Crefname{assumption}{Assumption}{Assumptions}

	\crefname{remark}{remark}{remarks}
	\Crefname{remark}{Remark}{Remarks}

	\crefname{example}{example}{examples}
	\Crefname{example}{Example}{Examples}

	\crefname{claim}{claim}{claims}
	\Crefname{claim}{Claim}{Claims}

	\crefname{equation}{equation}{equations}
	\Crefname{equation}{Equation}{Equations}

	\crefname{section}{section}{sections}
	\Crefname{section}{Section}{Sections}

	\crefname{subsection}{subsection}{subsections}
	\Crefname{subsection}{Subsection}{Subsections}

	\renewenvironment{proof}[1][\proofname]{\par\noindent\textbf{#1.} }{\qed}

	\numberwithin{equation}{section}

	\newcounter{FIGURE}
	\renewcommand{\theFIGURE}{\Roman{FIGURE}}
	\newcommand{\croch}[1]{[#1]}						
	\newcommand{\Croch}[1]{\left[#1\right]}				
	
	\newcommand{\Prth}[1]{\left(#1\right)}				
	\newcommand{\prth}[1]{(#1)}							
	
	\newcommand{\glmt}[1]{``#1''}						
	
	\newcommand{\acco}[1]{\lbrace#1\rbrace}				
	
	\newcommand{\prtH}[1]{\big(#1\big)}					
	\newcommand{\prtHH}[1]{\Big(#1\Big)}				
	\newcommand{\prtHHH}[1]{\bigg(#1\bigg)}				
	
	\newcommand{\crocHH}[1]{\Big[#1\Big]}				
	\newcommand{\crocHHH}[1]{\bigg[#1\bigg]}			
	
	\newcommand{\verti}[1]{\left\vert #1 \right\vert}
	\newcommand{\vertii}[2]{\Vert #1 \Vert_{#2}}
	\newcommand{\Vertii}[2]{\left\Vert #1 \right\Vert_{#2}}
	
	\newcommand{\vertI}[1]{\big\vert#1\big\vert}
	\newcommand{\vertII}[1]{\Big\vert#1\Big\vert}
	\newcommand{\vertIII}[1]{\bigg\vert#1\bigg\vert}

	\newcommand{\intervalleff}[2]{[#1,#2]}
	\newcommand{\intervallefo}[2]{[#1,#2)}
	\newcommand{\intervalleoo}[2]{(#1,#2)}
	\newcommand{\intervalleof}[2]{(#1,#2]}

	\newcommand{\gap}{1mm} 		
	\newcommand{\gapp}{1mm} 	
	
	\newcommand{\point}{\includegraphics[scale=1]{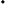}}

		\newcommand{\TITLE}{A piston to counteract diffusion: \break The influence of an inward-shifting boundary \break on the heat equation in half-space}
		\title{\textbf{\TITLE}}
		\edef\mytitle{\TITLE}
		\hypersetup{pdftitle={\mytitle}}

		\newdateformat{monthyeardate}{\monthname[\THEMONTH] \THEYEAR} 
		\date{\vspace{0.8mm}\itshape\monthyeardate\today} 				

		\author{%
    \href{https://www.samueltreton.fr/english/}{Samuel {\sc Tréton}}\thanks{%
        CNRS, UMR 6205, 
        \href{https://www.univ-brest.fr/laboratoire-mathematiques-bretagne-atlantique/en}{Laboratoire de Mathématiques de Bretagne Atlantique (LMBA)},\\
		\phantom{\hspace{4.5mm}}Université de Bretagne Occidentale, 6 avenue Victor Le Gorgeu, 29200 Brest, France.\\[-2.5mm]
    }%
    \and
    \href{https://sites.google.com/view/mingminzhang/home}{Mingmin {\sc Zhang}}\thanks{%
        CNRS, UMR 5219,
        \href{https://www.math.univ-toulouse.fr/fr/}{Institut de Mathématiques de Toulouse (IMT)},\\
		\phantom{\hspace{4.5mm}}Université de Toulouse, F-31062 Toulouse Cedex 9, France.\\[-5.5mm]
    }
}
		\hypersetup{pdfauthor={Samuel Tréton and Mingmin Zhang}}

\makeatletter
\renewcommand{\@fnsymbol}[1]{\ensuremath{\ifcase#1\or 1\or 2\or 3\or 4\or 5\or 6\or 7\or 8\or 9\else\@ctrerr\fi}}
\makeatother

\begin{document}

\maketitle
\thispagestyle{TitleStyle} 

\setcounter{footnote}{2} 

\vspace{-4.5mm}

\begin{abstract}
	\vspace{1.3mm}
	To better understand how populations respond to dynamic external pressure, we propose a new diffusion model in the moving half-line $\{z\geq b(t)\}$, where the boundary position $b\prth{t}$ is a given nondecreasing function of time.
	A Robin boundary condition is imposed at $z=b(t)$ to prevent individuals from leaving the domain, so that the shifting boundary acts as an impermeable wall---a \glmt{piston}---that sweeps the individuals it encounters.
	Our analysis focuses on the cases where $b(t)\sim ct^\beta$ with $\beta\in\intervalleff{0}{1}$.
	We prove quantitative convergence results characterized by attraction toward self-similar profiles, based on entropy techniques and Duhamel's principle.
	When $\beta$ goes through the critical value $1/2$, the shape of the self-similar asymptotic profile switches from Gaussian to exponential. In particular, this profile turns out to be stationary when $\beta=1$, reflecting a delicate balance between diffusion and advection induced by the moving boundary.
\end{abstract}

\vspace{2mm}

\noindent{\textbf{Keywords:}
        moving boundary,
		Fokker--Planck equations,
        asymptotic behavior,
		self-similar rescaling,
        entropy method,
        semigroup approach.
		}
\hypersetup{pdfkeywords={moving boundary, Fokker–Planck equations, asymptotic behavior, self-similar rescaling, entropy method, semigroup approach.}}

\vspace{2.5mm}

\noindent{\textbf{MSC2020:}
		\pdftooltip{35K05}{PDEs/Parabolic/heat Equation},
		\pdftooltip{35B40}{PDEs/Qualitative properties/Asymptotic behavior},
		\pdftooltip{35B33}{PDEs/Qualitative properties/Critical exponent},
		\pdftooltip{92D25}{Biology and other natural sciences/Population dynamics (general)}.
		}


\hypertarget{toc}{\tableofcontents}

\section{Introduction}\label{S1_intro}

\subsection{The model}

In this paper, we are interested in the following diffusion equation posed on the moving half-line
$\acco{z \geq b\prth{t}}$:
\begin{equation}\label{EQ_global_problem}
	\left\lbrace \begin{array}{llll}
		\partial_{t}u = d\partial_{zz} u, & \qquad &
		t>0, & z>b\prth{t},\\[0.9mm]
		-d\partial_{z}u = b'\prth{t}u , & \qquad &
		t>0, &z=b\prth{t},
	\end{array} \right .
\end{equation}
associated with bounded, compactly supported, and nonnegative initial data $u_{0}$.
Here, the constant $d>0$ denotes the diffusivity rate, and the position of the moving boundary $z = b\prth{t} \in C^{1}\prth{\mathbb{R}_{+}}$ is chosen as a nondecreasing function vanishing at $t=0$, representing an inward-shifting obstacle that sweeps individuals in the context of population dynamics.

\refstepcounter{FIGURE}
\label{FIG_1}
\begin{center}
\includegraphics[scale=1]{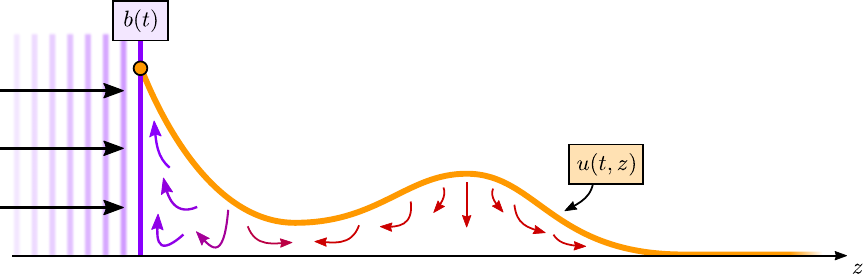}\\[1.5mm]
\begin{minipage}{133mm}
\begin{footnotesize}
	\textsc{Figure \theFIGURE~--- Variational trends induced by the moving boundary in equation \eqref{EQ_global_problem}.}
	\emph{Far from the boundary $z=b(t)$, individuals spread according to the diffusion process driven by the Laplacian (red arrows).
	Near the boundary, individuals are blocked and swept to the right, leading to concentration of population in this region (purple arrows).
	}
\end{footnotesize}
\end{minipage}
\end{center}

In \eqref{EQ_global_problem}, the Robin boundary condition ensures the conservation of mass within the moving domain $\acco{z \geq b\prth{t}}$. Indeed, defining the total mass at time $t\geq 0$ as
$${M\prth{t} : = \int_{b\prth{t}}^{\infty} u\prth{t,z} \, \md z},$$
and differentiating with respect to time yields
$$
0 = M'\prth{t} = 
- b'\prth{t} u\prth{t,b\prth{t}} - d \partial_{z}u\prth{t,b\prth{t}}.
$$
This boundary condition precisely compensates for the mass flux induced by the motion of the domain, thereby maintaining the physical consistency of the PDE dynamics.
In this sense, the Robin condition plays a corrective role by balancing the inward diffusive flux $-d\partial_{z}u$ with the effective transport flux $b'(t)u$ resulting from the boundary motion.
Throughout the manuscript, the constant $M$ thus denotes to the total mass of the population:
\begin{equation}\label{DEF_M}
	M : = \int_{0}^{\infty}u_{0}(z) \, \md z.
\end{equation}

In this paper, we aim to understand the asymptotic behavior of the solution with respect to the speed $b'(t)$ of the moving boundary.
Specifically, we consider an algebraic form for $b\prth{t}$, that is,
\begin{equation}\label{EQ_def_algebraic_b}
	b\prth{t} : = c\croch{\prth{1+t}^{\beta} - 1},
\end{equation}
for some constants $c>0$ and $\beta\geq 0$.
Observe that $b$ vanishes at $t=0$, its derivative $b'$ is locally bounded on $\mathbb{R}_{+}$, and $b\prth{t}$ behaves like $ct^{\beta}$ as $t\to\infty$.

To build intuition about this algebraic choice, let us first consider the simple case $\beta=0$, corresponding to the static case $b\equiv 0$ and thus to the classical heat equation on $\mathbb{R}_{+}$ with homogeneous Neumann boundary condition at $z=0$:
\begin{equation}\label{EQ_global_problem_Neumann}
\left\lbrace \begin{array}{llll}
	\partial_{t}u = d\partial_{zz} u, & \qquad &
	t>0, & z>0,\\[0.9mm]
	-d\partial_{z}u = 0 , & \qquad &
	t>0, &z=0.
\end{array} \right .
\end{equation}
By classical reflection arguments \cite{EvansPartial10}, the solution to \eqref{EQ_global_problem_Neumann} can be written as
\begin{equation}\label{EQ_HEAT_Neumann}
	u\prth{t,z} = \int_{\xi=0}^{\infty} H\prth{t,z,\xi} \, u_{0}\prth{\xi} \, \md \xi,
\end{equation}
where the fundamental solution $H$ is given by
\begin{equation}\label{EQ_fundamental_solution_Neumann}
	H\prth{t,z,\xi} := \frac{1}{\sqrt{4\pi dt}}\Croch{\exp\Prth{-\frac{\prth{z-\xi}^{2}}{4dt}}+\exp\Prth{-\frac{\prth{z+\xi}^{2}}{4dt}}}.
\end{equation}

From a microscopic perspective, the function $z\mapsto H(t,z,\xi)$ can be interpreted as the probability density function describing the position $Z_{t}$ of a single individual at time $t$, starting from the initial position $\xi\geq 0$.
This stochastic interpretation reflects the diffusion process at the individual scale, which, when aggregated over the entire population, yields the macroscopic diffusion equation \eqref{EQ_global_problem_Neumann} governing the population density.

By focusing on an individual starting from $\xi=0$, we can then compute its expected position as a function of time:
\begin{equation}
	\mathbb{E}\croch{Z_{t}} =
	\int_{z=0}^{\infty} z \, H\prth{t,z,0} \, \md z
	\stackrel{\eqref{EQ_fundamental_solution_Neumann}}{=}
	\frac{2}{\sqrt{4\pi dt}} \int_{z=0}^{\infty} z \, e^{-\frac{z^{2}}{4dt}} \, \md z =
	k \sqrt{t},\label{EQ_expectation}
\end{equation}
for some positive constant $k$.
The calculation \eqref{EQ_expectation} can be generalized to any initial position $\xi\geq 0$, and shows that the mean displacement of individuals increases proportionally to $\sqrt{t}$---which is characteristic of standard diffusive dispersal \cite[Chapter 1, Section 5]{CantrellSpatial04}.

The above observation suggests that the moving boundary $b(t)$ must advance at a rate comparable to or faster than $\sqrt{t}$ to significantly influence the population distribution.
Indeed, if the boundary moves at a slower rate, the typical scale of individual displacement eventually exceeds the boundary’s advance, resulting in only a weak interaction over time%
\footnote{Note that Brownian motion is recurrent in one dimension, meaning that individuals always have a nonzero probability of returning near the boundary. In particular, even when the boundary advances slowly, interaction with it persists at the individual level. However, as time grows, a larger fraction of the population spreads beyond the boundary’s influence, making the interaction effectively \glmt{weak} at the population scale.},
see \hyperref[FIG_Brownian_1]{Figure~\ref*{FIG_Brownian_1}~(a)(b)}.
Conversely, if the boundary moves at a faster rate than $\sqrt{t}$, then individuals remain confined near it, thus maintaining a strong interaction over time\footnote{In this regime (where the boundary outpaces the diffusion scaling) an interesting probabilistic question would be to quantify how the number of particle-boundary collisions grows over time.}, see \hyperref[FIG_Brownian_1]{Figure~\ref*{FIG_Brownian_1}~(c)}.

This reasoning motivates the choice of an algebraic boundary $b(t)\sim ct^{\beta}$ as $t\to\infty$, in line with \eqref{EQ_def_algebraic_b}.
By varying the exponent $\beta$, we can thus explore different regimes of boundary movement relative to the diffusive spread of the population.

\renewcommand{\gap}{1}
\refstepcounter{FIGURE}
\label{FIG_Brownian_1}
\begin{center}
	\includegraphics[scale=\gap]{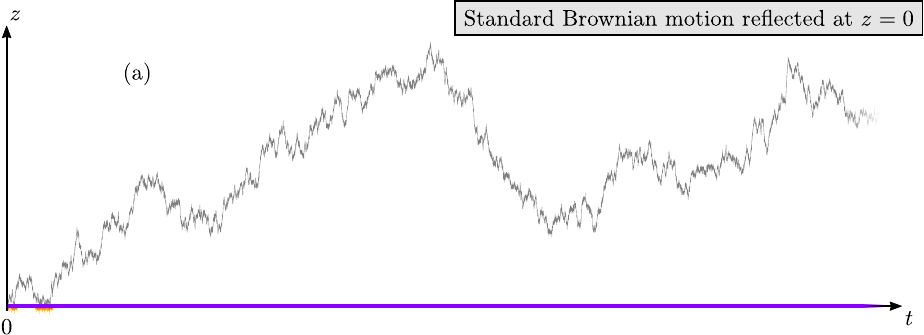}\\[3mm]
	\includegraphics[scale=\gap]{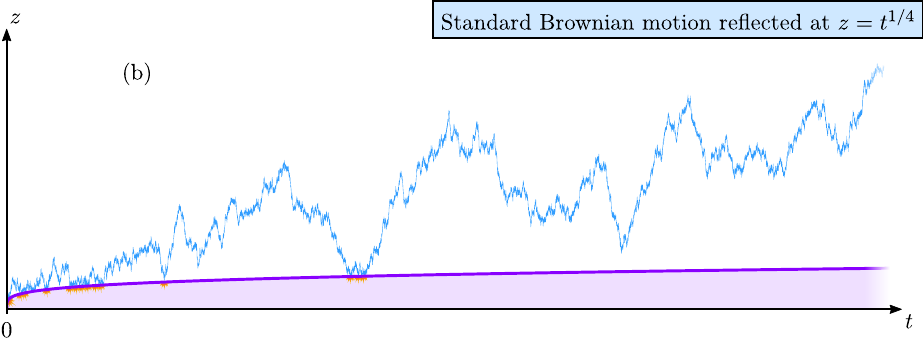}\\[3mm]
	\includegraphics[scale=\gap]{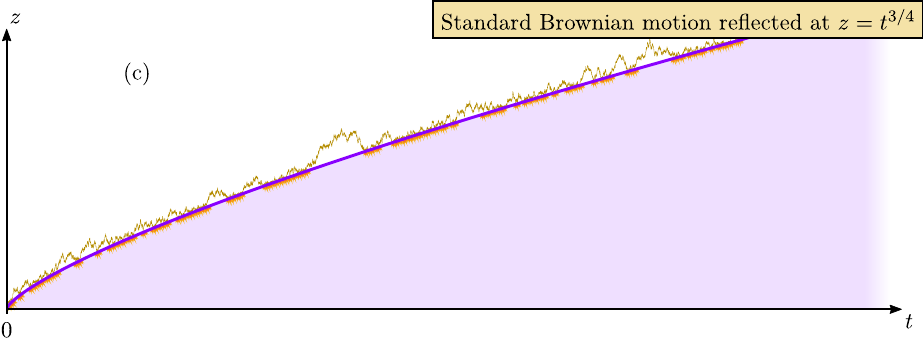}\\[1.5mm]
	\begin{minipage}{140mm}
	\begin{footnotesize}
	\textsc{Figure \theFIGURE~--- Illustration of typical microscopic trajectories for problem \eqref{EQ_global_problem} under different boundary speeds: static \textnormal{(a)}, slower than $\sqrt{t}$ \textnormal{(b)}, and faster than $\sqrt{t}$ \textnormal{(c)}.}
	\emph{
	Each simulation shows the trajectory of a single individual initially located at $z=0$, undergoing one-dimensional Brownian motion reflected at the moving boundary $z=b(t)$ (purple lines).
    The white regions above the purple lines correspond to the domain of interest $\{z\geq b(t)\}$.
	In cases \textnormal{(a)} and \textnormal{(b)}, interaction with the boundary remains weak, since the typical diffusive displacement scales like $\sqrt{t}$.
	In contrast, in case \textnormal{(c)}, frequent collisions with the boundary occur, leading to significant deviations in the trajectory.
	}
	\end{footnotesize}
	\end{minipage}
\end{center}

\subsection{Mathematical background and biological context}

Since the seminal works of Fisher \cite{FisherWave37} and Kolmogorov, Petrovski, and Piskunov \cite{KolmogorovStudy37} in the 1930s, reaction--diffusion equations have received significant attention within the mathematical community. They are valued not only for their extensive applications across diverse fields such as ecology, epidemiology, and social sciences, but also for their relative accessibility in analyzing certain qualitative behaviors.

Over the last few decades, a rich literature has developed around reaction--diffusion equations in moving environments (e.g.,
\cite{AlfaroEffect17,BerestyckiCan14,BerestyckiCan09,BerestyckiReactiondiffusion08,BouhoursSpreading19,FigueroaIglesiasLong18,FigueroaIglesiasSelection21,RoquesAdaptation20}).
In these models, species typically track habitats moving at a constant speed---a widely adopted paradigm for modeling phenomena such as poleward species shifts under climate change.
Environmental changes are usually incorporated either through an advection term, resulting in equations of the form
$$\partial_{t}u = \partial_{zz}u - c\partial_{z} u + f(z,u),$$
or by evaluating the growth term in a co-moving frame, namely $f(z - ct, u)$.

In contrast, our contribution departs from the classical paradigm in two key respects.
First, we prescribe algebraically growing boundary $b(t)=c\croch{\prth{1+t}^{\beta} - 1}$ which, to the best of our knowledge, has not been explored in the literature.
Unlike the constant-speed shifts of earlier models, this time-dependent boundary induces accelerating or decelerating boundary motion, offering new insights into population adaptation under time-varying spatial constraints.
Second, our work belongs to the emerging framework of dynamically prescribed domains $\Omega(t)$. This formulation has only recently gained attention, notably through the works of Allwright \cite{AllwrightExact22,AllwrightReaction23}, where reaction--diffusion equations were studied in time-dependent intervals with homogeneous Dirichlet boundary conditions.
Here, we move to a qualitatively different geometry by considering a time-dependent half-space\footnote{Owing to its fundamental structure, the half-space is the only connected open subset of $\mathbb{R}$ possessing a single boundary point. This property eliminates the interference that would arise from the presence of a second boundary, making the half-space a natural setting for isolating and quantifying population--boundary interactions.} and focus on the purely diffusive case, a combination that appears absent from the existing literature.

Moreover, our setting fundamentally differs from free boundary problems, where the domain evolution is governed by the solution itself through interface conditions or conservation laws. Instead, in our case, the boundary motion is externally prescribed---fully independent of the solution---and thus acts as an exogenous constraint on the population.

Altogether, our framework gives rise to qualitatively distinct dynamics: species must cope with a nonlinearly receding environment rather than follow a linearly translating habitat.
Far from being purely theoretical, these dynamics have concrete counterparts in real-world ecological settings.
In the context of population dynamics, our model can represent environments evolving under disruptive influences such as forest fires \cite{MalhiEcological04}, rising sea levels \cite{LoucksSea10,MenonPreliminary10}, melting ice \cite{MooreArctic08}, and declining snow cover \cite{CarrerRecent23}---all of which progressively erode the space available to biological communities.
In such scenarios, both the geometry and the location of the habitable space evolve over time, often in nontrivial ways.
Our setting, centered on a single moving boundary, captures the essence of these spatial constraints while remaining analytically tractable.
It thus provides an insightful theoretical lens for understanding the survival and dispersal challenges faced by species in receding habitats.


\section{Main results}\label{S2_main_results}

By performing the change of variable
\begin{equation}\label{EQ_change_of_var_u_to_v}
	v\prth{t,x} : = u\prth{t,x+b\prth{t}},
\end{equation}
we reformulate \eqref{EQ_global_problem} as the following problem
\begin{equation}\label{v-eqn}
	\left\lbrace \begin{array}{llll}
		\partial_{t}v = d\partial_{xx} v + b'\prth{t}\partial_{x}v, & \qquad &
		t>0, & x>0,\\[0.9mm]
		-d\partial_{x}v = b'\prth{t}v, & \qquad &
		t>0, & x=0,
	\end{array} \right .
\end{equation}
posed on the static half-space $\mathbb{R}_{+}$, with initial condition
$v|_{t=0} = u_{0} = : v_{0}$.
From now on, the discussion is framed entirely in terms of the transformed variable $v$ rather than the original function $u$.

Assuming $v_0$ is bounded, compactly supported, and nonnegative in $\mathbb{R}_{+}$---or more generally belongs to $L^1\prth{\mathbb{R}_{+}} \cap L^\infty\prth{\mathbb{R}_{+}}$---the global well-posedness of \eqref{v-eqn} follows from classical results in nonautonomous semigroup theory
\cite{PazySemigroups83},
which guarantee the existence and uniqueness of a mild solution
$v\in{C}\prth{\intervallefo{0}{\infty}; L^1\prth{\mathbb{R}_{+}} \cap L^\infty\prth{\mathbb{R}_{+}}}$.
Furthermore, standard parabolic regularity theory implies that this mild solution is, in fact, smooth for positive times: $v\in C^{1,2}_{t,x}\prth{\intervalleoo{0}{\infty}\times\mathbb{R}_{+}}$, so that $v$ is a classical solution to \eqref{v-eqn} for all $t>0$.
Finally the problem satisfies the comparison principle, which ensures in particular that $v$ remains nonnegative whenever $v_{0}$ is nonnegative.

\medskip

Our first main result provides an explicit formula for the solution $v$ to \eqref{v-eqn} when $\beta=1$. This case corresponds to the linear regime of the moving boundary, where the boundary position grows linearly in time---namely $b\prth{t} = ct$. This explicit representation of $v$ enables a sharp analysis of its long-time behavior, revealing fast convergence toward a nontrivial steady state.

\begin{theorem}[Fundamental solution and asymptotic behavior for $\beta=1$]
\label{TH_fundamental_sol_for_beta_1}
    Assume that $c>0$ and $b\prth{t}=ct$. Let $v$ be the unique solution to \eqref{v-eqn} starting from bounded, compactly supported, and nonnegative initial condition $v_{0}$.
	Then the two following points hold true:
 \begin{itemize}
     \item[(i)] For any $t>0$ and $x\in\mathbb{R}_{+}$, the solution $v$ is explicitly given by
\begin{equation}\label{EQ_fundamental_sol_v_linear}
		v\prth{t,x} = \int_{\xi=0}^{\infty}
		H\prth{t,x,\xi}\,v_{0}\prth{\xi} \, \md \xi,
	\end{equation}
	where the fundamental solution $H$ can be written in terms of the well-known heat kernel,
	$G\prth{t,X} = e^{-X^{2}/\prth{4dt}}/\sqrt{4\pi dt}$,
	as
\begin{equation}\label{EQ_heat_kernel_linear}
	\hspace{-6.0mm}
	H\prth{t,x,\xi} : = G(t, x + ct - \xi) +
	G(t, x + ct + \xi) e^{\frac{c}{d}\xi} +
	\frac{c}{d} \int_{\omega=\xi}^{\infty} G(t, x + ct + \omega) e^{\frac{c}{d}\omega} \, \md \omega.
	\end{equation}
    \item[(ii)] For any $x\in \mathbb{R}_{+}$, let
\begin{equation}\label{EQ_def_V_beta_equals_1}
		V\prth{x} := M\frac{c}{d}
        \,
        e^{-\frac{c}{d}x},
	\end{equation}
	then, there are two positive constants $\ell$ and $k$, such that
\begin{equation}\label{EQ_cv_to_V_when_beta_equals_1}
		\vertii{
			v\prth{t,x} - V\prth{x}
		}{L^{1}_{x}\prth{\mathbb{R}_{+}^{}}}
		\leq
		\ell e^{-k t},
		\qquad
		\forall t>0.
	\end{equation}
     \end{itemize}
\end{theorem}


The construction of the solution $v$ in \hyperref[TH_fundamental_sol_for_beta_1]{Theorem~\ref*{TH_fundamental_sol_for_beta_1}~\emph{(i)}} relies on an extension argument made possible by the fact that \eqref{v-eqn} is autonomous when $\beta=1$.

In \eqref{EQ_fundamental_sol_v_linear}-\eqref{EQ_heat_kernel_linear},
the fundamental solution $H$ has an insightful stochastic interpretation.
Indeed, $H\prth{t,x,\xi}$ represents the probability density of the trajectory $X_{t}$ of an individual undergoing reflected Brownian motion with constant drift on $\mathbb{R}_{+}$, starting from $\xi\geq 0$.
This setting is analogous to that described by the probability kernel \eqref{EQ_fundamental_solution_Neumann} associated with the Neumann heat problem \eqref{EQ_global_problem_Neumann}.

Beyond this interpretation, the explicit representation of the solution enables a detailed quantitative analysis. \hyperref[TH_fundamental_sol_for_beta_1]{Theorem~\ref*{TH_fundamental_sol_for_beta_1}~\emph{(ii)}} establishes the $L^1$-convergence of the solution $v\prth{t,\point}$ toward a nontrivial stationary state $V$ at an exponential rate, which is particularly remarkable given that diffusion-driven problems typically exhibit polynomial convergence \cite{VazquezAsymptotic18}.
This rapid stabilization highlights that the moving boundary not only induces mass accumulation but also accelerates convergence toward the stationary profile by confining the individuals.
It is worth noting that such an exponential convergence rate is typical for the heat equation on bounded intervals with Neumann boundary conditions.
In our setting, the confinement induced by the moving boundary then appears to recreate such a bounded-domain-like effect on the population.

It is also notable that the steady-state solution 
$V$ is nontrivial, indicating that the population does not vanish (as it does for classical diffusion equations on unbounded domains), but instead persists and stabilizes into a stationary exponential distribution.
This phenomenon reflects a delicate balance between two opposing mechanisms: the tendency of individuals to disperse through diffusion, and their accumulation near the boundary due to the advection induced by its motion.
The following heuristic perspective helps to further explain this balance: by splitting the dynamics arising from \eqref{v-eqn}$|_{b(t)=ct}$ into a diffusive part:
\begin{equation}\label{EQ_split_v_eqn_diff}
	\partial_{t}v_{\text{diff}} = d\partial_{xx} v_{\text{diff}},
\end{equation}
and an advective part:
\begin{equation}\label{EQ_split_v_eqn_adv}
	\partial_{t}v_{\text{adv}} = c\partial_{x} v_{\text{adv}},
\end{equation}
both associated with the same initial condition
$$v_{\text{diff}}|_{t=0}\prth{x} = v_{\text{adv}}|_{t=0}\prth{x} = e^{-\frac{c}{d}x},$$
we observe that the resulting solutions are exponentials translating in opposite directions:
$$
v_{\text{diff}}\prth{t,x} = e^{-\frac{c}{d}\prth{x-ct}}
\qquad
\text{and}
\qquad
v_{\text{adv}}\prth{t,x} = e^{-\frac{c}{d}\prth{x+ct}}.
$$
This symmetry elegantly reveals the natural compensation between diffusion and drift effects, leading to convergence toward the exponential stationary state $V(x)=\frac{Mc}{d}\sqrt{v_{\text{diff}}\prth{t,x}\cdot v_{\text{adv}}\prth{t,x}}$.

Finally, \Cref{TH_fundamental_sol_for_beta_1} provides a guiding framework for understanding more general regimes.
When the boundary speed is sublinear ($\beta<1$), the solution is expected to exhibit dissipative behavior, as carefully discussed in \Cref{PROPO_BETA_0} and \Cref{TH_asymptotic_behavior_solutions}.
In contrast, a superlinear boundary motion ($\beta>1$) is anticipated to induce strong concentration effects, whose detailed analysis is deferred to future work due to intricate analytical challenges (see \Cref{R_beta_greater_1}).

\medskip

We now proceed with our second main result, describing the convergence of the solution to \eqref{v-eqn} toward self-similar profiles, for any $\beta\in\intervalleff{0}{1}$.
Let us start with the particular case $\beta=0$, in which \eqref{v-eqn} reduces to the classical Neumann heat equation \eqref{EQ_global_problem_Neumann} in $\mathbb{R}_{+}$.
The following proposition, well-known in the literature (see \cite[Theorem 4.1]{VazquezAsymptotic18} for instance), describes the corresponding sharp asymptotic behavior of the solution.
\begin{proposition}[Asymptotic behavior for $\beta=0$]\label{PROPO_BETA_0}
     Let $v$ be the unique solution to \eqref{EQ_global_problem_Neumann} starting from bounded, compactly supported, and nonnegative initial condition $v_{0}$.
    Then there exists $\ell>0$ such that
		\begin{equation}
		\label{CONVERGENCE_BETA_0}
			\Vertii{
				v\prth{t,x} -
				\frac{1}{\sqrt{1+t}}
				W\prtHH{\frac{x}{\sqrt{1+t}}}
			}{L^{1}_{x}\prth{\mathbb{R}_{+}^{}}}
			\leq \frac{\ell}{\sqrt{1+t}},
			\qquad
			\forall t>0,
		\end{equation}
		where $W(y) = W\prth{0} \exp\!\Prth{-\frac{y^{2}}{4d}}$ for any $y\in\mathbb{R}_{+}$,
        with $W\prth{0}\geq 0$ uniquely determined so that the total mass of the self-similar profile $W$ equals $M$.
\end{proposition}

Building upon the case $\beta=0$, we extend the analysis to general $\beta\in\intervalleof{0}{1}$, which leads to our main asymptotic result.

\begin{theorem}[{Asymptotic behavior for $\beta\in\intervalleof{0}{1}$}]\label{TH_asymptotic_behavior_solutions}
	Assume that $b\prth{t} = c\croch{\prth{1+t}^{\beta} - 1}$ with $c>0$ and $\beta\in \intervalleof{0}{1}$.
Let $v$ be the unique solution to \eqref{v-eqn} starting from bounded, compactly supported, and nonnegative initial condition $v_{0}$.
Then the asymptotic behavior of $v$ is characterized by convergence to self-similar profiles, according to the three following regimes:
	\begin{itemize}
		\item[(i)] If $0 < \beta<1/2$ ({sub-critical regime}), then there exists $\ell>0$ such that
		\begin{equation}
		\label{beta<1/2: entropy + Duhamel}
			\Vertii{
				v\prth{t,x} -
				\frac{1}{\sqrt{1+t}}
				W\prtHH{\frac{x}{\sqrt{1+t}}}
			}{L^{1}_{x}\prth{\mathbb{R}_{+}^{}}}
			\leq {\frac{\ell}{(1+t)^{\frac{1}{4}-\frac{\beta}{2}}}},
			\qquad
			\forall t>0,
		\end{equation}
		where $W(y) = W\prth{0} \exp\!\Prth{-\frac{y^{2}}{4d}}$ for any $y\in\mathbb{R}_{+}$.
		\item[(ii)] If $\beta=1/2$ ({critical regime}), then there exists $\ell>0$ such that
		\begin{equation}
			\label{beta=1/2: entropy conclu}
				\Vertii{
					v\prth{t,x} -
					\frac{1}{\sqrt{1+t}}
					W\prtHH{\frac{x}{\sqrt{1+t}}}
				}{L^{1}_{x}\prth{\mathbb{R}_{+}^{}}}
				\leq {\frac{\ell}{\sqrt{1+t}}},
				\qquad
				\forall t>0,
		\end{equation}
		where $W(y) = W\prth{0} \exp\!\Prth{-\frac{y^{2}}{4d} - \frac{cy}{2d}}$ for any $y\in\mathbb{R}_{+}$.
		\item[(iii)] If ${1/2<\beta\leq 1}$ ({super-critical regime}), then there exists $\ell>0$ such that
		\begin{equation}\label{EQ_asymptotic_behavior_solutions_larger_half}
			\Vertii{
				v\prth{t,x} -
				\frac{b'\prth{t}}{c\beta}
				W\prtHHH{\frac{b'\prth{t}}{c\beta}\,x}
			}{L^{1}_{x}\prth{\mathbb{R}_{+}^{}}}
			\leq {\frac{\ell \log\prth{1+t}}{\prth{1+t}^{\beta-\frac{1}{2}}}},
			\qquad
			\forall t>1,
		\end{equation}
		where $W(y) = W\prth{0} \exp\!\Prth{-\frac{c\beta}{d} y}$ for any $y\in\mathbb{R}_{+}$.
	\end{itemize}
	In each case above, $W\prth{0}\geq 0$ is the uniquely determined constant such that the total mass of the self-similar profile $W$ equals $M$.
\end{theorem}

\Cref{TH_asymptotic_behavior_solutions} reveals three distinct regimes, governed by the value of the parameter $\beta$, as further illustrated in \hyperref[FIG_2]{Figure~\ref*{FIG_2}}.
One of the main ingredients of its proof is to introduce a general self-similar transformation of the form
\begin{equation}\label{EQ_RESCALING}
	v(t,x) = f(t)\cdot w\prtH{g(t),f(t)x},
\end{equation}
where the scaling functions $f(t)$ and $g(t)$ are determined according to the different regimes, see \Cref{S3_diffusive_regime} and \Cref{S5_super_critical_case}. 
Under such scaling, the problem satisfied by $w$ becomes a Fokker--Planck type equation:
\begin{equation}\label{EQ_Fokker_Planck}
	\begin{array}{llll}
	\partial_{\tau} w = \partial_{y}\Croch{
		D \partial_{y} w + \partial_{y}\Phi\prth{\tau,y} w
	}, & \, &
	\tau > 0, & y>0,
	\end{array}
\end{equation}
where the time-dependent potential $\Phi(\tau,y)$ stabilizes over time toward a stationary potential $\Phi_{\infty}(y)$. Depending on the parameter $\beta$, this stationary potential $\Phi_{\infty}$ is either quadratic or linear in $y$, and the stationary solutions of this resulting problem \eqref{EQ_Fokker_Planck} are given by:
$$
W\prth{y} = W\prth{0} e^{-\frac{1}{D}\Phi_{\infty}\prth{y}},
$$
consequently leading to Gaussian, exponential, or mixed profiles, according to the precise form of $\Phi_{\infty}$.
It is worth noting that $f\prth{t}\equiv 1$ and $g(t)=t$ in the linear regime $\beta=1$, so that the self-similar transformation \eqref{EQ_RESCALING} reduces to the identity and the unknown $v$ and $w$ coincide.

\medskip

As anticipated by the stochastic perspective introduced in \Cref{S1_intro}, the model displays a sharp transition at the critical value $\beta=1/2$. At this threshold, the boundary drift is of magnitude $\mathcal{O}\prth{\sqrt{t}}$, exactly matching the mean displacement of a reflected Brownian particle on $\mathbb{R}_{+}$.
This alignment marks the onset of significant boundary effects.
In particular, it triggers a self-similar intermediate profile, characterized by a Gaussian shape modulated by an exponential factor, reflecting the emergence of a competition between diffusion and boundary-driven accumulation.
In the sub-critical regime $\beta \in\intervalleoo{0}{\frac{1}{2}}$, the self-similar profiles adopt Gaussian shapes---exactly as for the solutions to the classical heat equation on $\mathbb{R}_{+}$ with Neumann boundary condition, see \Cref{PROPO_BETA_0}. This point is consistent with the weak interactions between the particles and the boundary discussed in \Cref{S1_intro}, see \hyperref[FIG_Brownian_1]{Figure~\ref*{FIG_Brownian_1}~(a)(b)}.
By contrast, in the super-critical regime $\beta\in(\frac{1}{2},1]$, the boundary outpaces the diffusive scaling, giving rise to exponential self-similar profiles with a noticeably heavier tail.
This highlights the strong interaction between the moving boundary and the diffusion process, see \hyperref[FIG_Brownian_1]{Figure \ref*{FIG_Brownian_1} (c)}.

We also point out that, when $\beta=1$, \hyperref[TH_asymptotic_behavior_solutions]{Theorem~\ref*{TH_asymptotic_behavior_solutions}~\emph{(iii)}} provides a polynomial convergence rate, which is significantly less optimal than the sharp exponential convergence exhibited in \hyperref[TH_fundamental_sol_for_beta_1]{Theorem~\ref*{TH_fundamental_sol_for_beta_1}~\emph{(ii)}}.
When $\beta=1/2$, the convergence rate we obtain coincides with that of the Neumann heat equation (see \Cref{PROPO_BETA_0}), for which the entropy method is known to yield the sharp decay rate (see \Cref{REM_beta=0}). Since our proof strategy relies on the same type of entropy arguments---involving only the logarithmic Sobolev and Csisz\'ar--Kullback inequalities, without any crude estimates---it is reasonable to believe that the resulting decay rate is also sharp in this critical case.
For $\beta<1/2$, we do not expect better rates with our technique, as explained in \cref{SS33_beta_lower_1_2}---see
\eqref{EQ_entropy_diff_yet_sharp}-\eqref{3}.

Lastly, let us emphasize that, for $0 \leq \beta < 1$, the solution $v$ tends to vanish as a result of the decaying boundary speed $b'(t)$. Varying 
$\beta$ from $0$ to $1$ thus provides a clear depiction of how the influence of the moving boundary \emph{counteracts diffusion} up to the limiting case $\beta = 1$, where the balance is achieved between the dispersive effect of diffusion and the aggregative effect of the moving boundary.


\renewcommand{\gap}{1}
\refstepcounter{FIGURE}
\label{FIG_2}
\begin{center}
	\includegraphics[scale=\gap]{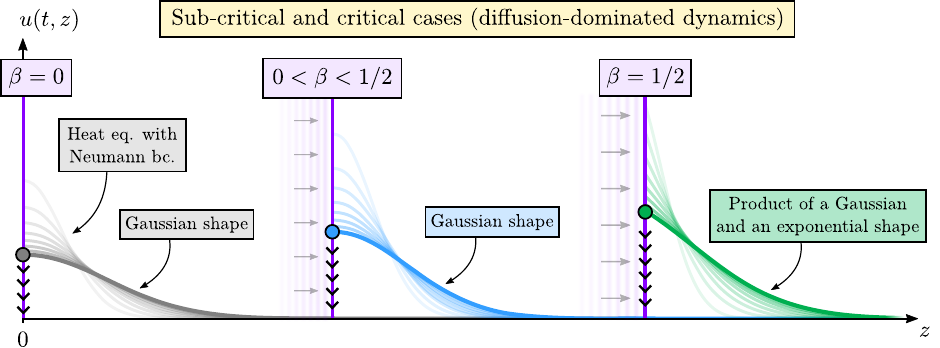}\\[3mm]
	\includegraphics[scale=\gap]{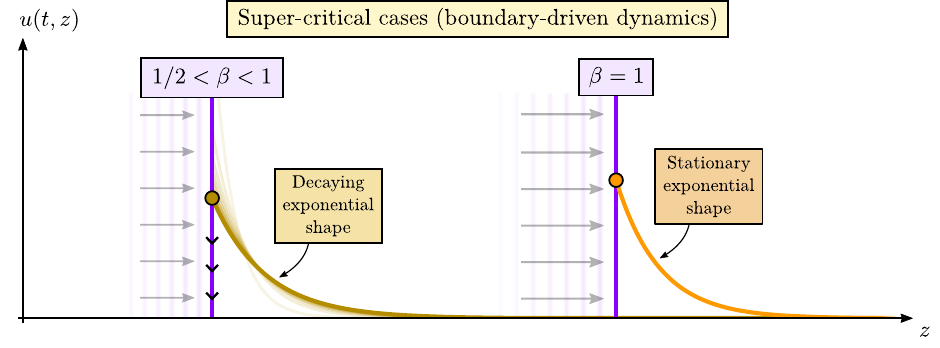}\\[0.5mm]
	\begin{minipage}{149mm}
	\begin{footnotesize}
	\textsc{Figure \theFIGURE~--- long-time behavior of the solutions to \eqref{EQ_global_problem} regarding different regimes of $\beta$.}
	\emph{We recall that the moving boundary $b(t)$ (depicted by the purple vertical lines) behaves asymptotically like $ct^{\beta}$ as $t\to \infty$.
A transition in the shape of the asymptotic profiles occurs at 
$\beta=1/2$, shifting from Gaussian to exponential.
	This threshold also marks a change in the $L^{\infty}$ decay rate, achieved by $u\prth{t,b\prth{t}}=v(t,0)$, shifting from $\mathcal{O}\prth{t^{-1/2}}$ for $\beta\in\intervalleff{0}{\frac{1}{2}}$ to $\mathcal{O}\prth{t^{\beta-1}}$ for $\beta\in\intervalleff{\frac{1}{2}}{1}$.
	As a complement to this macroscopic viewpoint, typical individual trajectories are illustrated in \hyperref[FIG_Brownian_1]{Figure~\ref*{FIG_Brownian_1}} for the cases $\beta \in \acco{0,\frac{1}{4},\frac{3}{4}}$, highlighting the increasing influence of the moving boundary at the microscopic scale.
	}
	\end{footnotesize}
	\end{minipage}
\end{center}

\begin{remark}[On the case $\beta>1$]
\label{R_beta_greater_1}
	It is worth noting that the super-critical self-similar rescaling \eqref{EQ_rescaling_super_sqrt} remains valid for $\beta>1$, and that the Cauchy problem \eqref{EQ_for_w_beta_larger_1/2} still governs the resulting function $w$ in this case.
	Stationary solutions with exponential decay therefore persist, exactly as in the intermediate regime $\beta\in\intervalleof{\frac{1}{2}}{1}$.
	Numerical experiments further suggest that, in this superlinear setting, $w\prth{t,\point}$ converges to such an exponentially decaying steady state.
	This behavior suggests that the function $v$ increasingly concentrates near the moving boundary over time, yet without finite-time blow-up.
	Specifically, we anticipate the following approximate long-time behavior:
	$$\vertii{v(t,\point)}{L^{\infty}(\mathbb{R}_{+})} \approx v(t,0) \approx \frac{b'(t)}{c\beta} = \prth{1+t}^{\beta-1}.$$

	A key point is that the coefficient $\eta(\tau)$ appearing in problem \eqref{EQ_for_w_beta_larger_1/2}, and given by \eqref{EQ_eta}:
    \begin{equation*}
	\eta\prth{\tau}
	=
	\frac{\beta-1}{1 + \prth{2\beta-1}\tau},
\end{equation*} 
	changes sign as $\beta$ crosses the value $1$.
    This sign change drastically alters the behavior of the drift term in \eqref{EQ_for_w_beta_larger_1/2}, and introduces new analytical challenges.
	To be more specific, when $\beta \leq 1$, the drift $c\beta-\eta(\tau)y$ consistently drives the population toward the boundary, thereby promoting mass concentration near $y=0$.
	In contrast, for $\beta>1$, the drift points toward the boundary only for small values of $y$ (where $c\beta-\eta(\tau)y$ remains nonnegative). As 
$y$ increases, the influence of 
$-\eta(\tau)y$ eventually dominates, causing the drift to reverse and effectively driving the population away from the boundary.
	This subtle interplay between boundary accumulation and long-range repulsion significantly complicates the analysis.
	
	In particular, a major obstacle lies in controlling the first moment of the function $w$, for which the standard techniques used for $\beta\in\intervalleof{\frac{1}{2}}{1}$ (such as the comparison principle and the construction of supersolutions) fail, since the reversed sign of $\eta(\tau)$ turns our supersolutions into subsolutions.
	As the first moment directly impacts the upper bound of the convolution term in \eqref{EQ_Duhamel_1}, it is natural to first explore whether the strategy developed for $\beta\in\intervalleof{\frac{1}{2}}{1}$ can be adapted to recover a suitable bound on that moment. Otherwise, a completely different approach will be required.
	Nonetheless, we believe that the first moment is decisive for the long-time dynamics, as clearly evidenced in related studies by Calvez \emph{et al.} \cite{CalvezAnalysis12} and Lepoutre \emph{et al.} \cite{LepoutreCell14}. In those works, explicit evolution equations for the first moment turned out to be significantly useful for the analysis of convergence and potential blow-up behaviors.
	In our model, however, the drift varies in both space and time, which prevents the direct application of these earlier strategies.
	
	Although the analytical framework developed here provides a solid foundation, a refined analysis is still needed to tackle the sign-changing drift and its influence on the first-moment dynamics.
	We leave this challenging question for future work, with the expectation that sharp bounds on the first moment will ultimately yield a complete description of the super-linear regime.
\end{remark}

\paragraph{Strategy of the proofs}

The proof of \hyperref[TH_fundamental_sol_for_beta_1]{Theorem~\ref*{TH_fundamental_sol_for_beta_1}~\emph{(i)}} is based on the explicit construction of the solution \emph{via} an extension method:
we extend the initial data $v_{0}$ to the whole real line, choosing the extension so that, under the evolution governed by $\partial_t = d \partial_{xx} + c \partial_x$, the resulting solution satisfies the Robin boundary condition at $x = 0$ for all $t > 0$.
We then recover the desired solution by restricting the extended solution to the positive half-line $\mathbb{R}_{+}$, and a detailed manipulation of integrals reveals the explicit form of the fundamental solution $H\prth{t,x,\xi}$, as given in \eqref{EQ_heat_kernel_linear}.

In turn, the explicit representation of the solution allows for a precise analysis of its long-time behavior, as stated in \hyperref[TH_fundamental_sol_for_beta_1]{Theorem~\ref*{TH_fundamental_sol_for_beta_1}~\emph{(ii)}}.
By carefully analyzing the structure of the fundamental solution, we derive $L^{1}$-estimates for $v\prth{t,\point}-V$ demonstrating exponential convergence toward the stationary profile.

\medskip

As already explained, \Cref{TH_asymptotic_behavior_solutions} is grounded on a self-similar change of variables---see \eqref{EQ_RESCALING}.
When $b(t)$ is of order $\mathcal{O}(\sqrt{t})$, namely $\beta \leq 1/2$, we use the parabolic scaling $\tau=\log\sqrt{1+t}$ and $y=\frac{x}{\sqrt{1+t}}$, and study the rescaled density $w$ satisfying
\begin{equation*}
	v(t,x)=
	\frac{1}{\sqrt{1+t}} \;
	w
	\Big(
		\log\sqrt{1+t},
		\frac{x}{\sqrt{1+t}}
	\Big).
\end{equation*}
The distance between $w$ and the associated stationary profile $W$ is then measured by means of the Boltzmann relative entropy \cite{VazquezAsymptotic18} defined by
\begin{equation*}
\mathcal{H}(w|W):=\int_0^\infty w(\tau,y)\log\frac{w(\tau,y)}{W(y)}\, \md y.
\end{equation*}
As in the case of the heat equation, this functional can be differentiated with respect to $\tau$.
The strict uniform convexity of the quadratic potential $\Phi(\tau,y)$ (see \eqref{EQ_Fokker_Planck}) then allows us to apply logarithmic Sobolev (\Cref{Lem_logSobolev}) and Csisz\'ar--Kullback (\Cref{Lem_Csiszar_Kullback}) inequalities to estimate the entropy dissipation ${\md\mathcal{H}}/{\md\tau}$ and establish convergence.

When $\beta>1/2$, the change of variables is more delicate:
\begin{equation*}\label{EQ_rescaling_super_sqrt_main_result}
	v\prth{t,x} : =
	\frac{b'\prth{t}}{c\beta} \,
	w \prtHHH{
		\int_{s=0}^{t} \crocHHH{\frac{b'\prth{s}}{c\beta}}^{2} ds , \;
		\frac{b'\prth{t}}{c\beta}x
	}.
\end{equation*}
In this super-critical regime, our entropy approach breaks down, since the potential $\Phi$ in the Fokker--Planck equation \eqref{EQ_Fokker_Planck} loses its strict convexity as $\tau\to\infty$ and becomes linear, thus preventing the use of the logarithmic Sobolev inequality.

Instead, our approach is based on the observation that the Fokker--Planck problem \eqref{EQ_for_w_beta_larger_1/2} governing $w$ can be regarded as the equation associated with $\beta=1$ plus a vanishing perturbation in the interior equation, which we treat as a source term using Duhamel’s principle \cite{GigaNonlinear10}.
The main difficulty then lies in showing that the convolution term in \eqref{EQ_Duhamel_1}, arising from the perturbative source in Duhamel’s formula, vanishes in $L^{1}$ as $\tau\to\infty$.

While this method ultimately allows us to establish the convergence of $w$ toward $W$, it crucially relies on an \emph{a priori} control of the first moment of the solution---which becomes the main obstruction when $\beta>1$ (see \Cref{R_beta_greater_1}).

\paragraph{Perspectives for future research}
In this paper, we focus specifically on the purely diffusive mechanism in the presence of a moving boundary. This simple setting exhibits a surprisingly rich variety of behaviors---as concisely illustrated in \hyperref[FIG_2]{Figure~\ref{FIG_2}}---and paves the way toward the study of more complex nonlinear and density-dependent models, such as reaction--diffusion equations:
\begin{equation}
	\left\lbrace \begin{array}{llll}
		\partial_{t}u = d\partial_{zz} u + f\prth{u}, & \qquad &
		t>0, & z>b\prth{t},\\[0.9mm]
		-d\partial_{z}u = b'\prth{t}u , & \qquad &
		t>0, &z=b\prth{t}.
	\end{array} \right .
\end{equation}

For instance, understanding the $L^{\infty}$ decay rate of the diffusive solutions is essential for studying the viability of biological populations when the reaction term takes the \emph{degenerate monostable} form \cite{AronsonMultidimensional78, FujitaBlowing66}
\begin{equation}\label{EQ_reaction_function}
	f(u) = u^{1+p}(1 - u), \qquad p\geq 0,
\end{equation}
This reaction function $f$ models a well-known ecological phenomenon of growth difficulty at low densities, known as the \emph{Allee effect} \cite{AlleeAnimal31}.\footnote{The exponent $p \geq 0$ quantifies the strength of the low-density growth limitation (\emph{Allee effect}), with $p = 0$ (KPP case) corresponding to the absence of such limitation.}

In the context of our dynamic domain, the speed of the moving boundary is expected to play a crucial role in determining population survival. In particular, it may induce \emph{ecological rescues} in the presence of an \emph{Allee effect}. In such cases, populations that would otherwise become extinct in static domains could survive by aggregating against the moving boundary, thus benefiting from enhanced growth conditions.

Conversely, a fast-moving boundary may instead become harmful to the population, amplifying \emph{intraspecific competition}. This phenomenon, captured by the negative term $-u^{2+p}$ in \eqref{EQ_reaction_function}, may ultimately drive the population to extinction, as a result of excessive aggregation of individuals in the vicinity of the boundary. Importantly, this risk persists even in the classical KPP case $f(u) = u(1 - u)$, i.e., when $p=0$ in \eqref{EQ_reaction_function}, in sharp contrast with the well-known \emph{hair-trigger effect}---namely the success of invasion in unbounded domains.

As a result, the relationship between boundary speed and population persistence may become non-monotonic, since extinction may occur at both very low and very high speeds, while intermediate speeds could allow for survival.
It is worth mentioning that similar non-monotonic phenomena between initial fragmentation and success rate of invasion in the presence of an \emph{Allee~effect} have already been documented in previous works \cite{GarnierSuccess12, AlfaroPropagation24}.

\medskip

Beyond deterministic dynamics, adopting a stochastic perspective on the underlying microscopic diffusive process could be highly valuable for better understanding the evolution of the first moment of the solution, i.e., the mean position of an individual. Such an approach would complement our deterministic analysis and may offer new insights into the dynamics, particularly in the superlinear regime $\beta > 1$---whose analytical understanding remains an open challenge (see~\Cref{R_beta_greater_1}).

\medskip

These directions naturally emerge from our analysis and will be the subject of continued investigation.

\paragraph{Outline of the paper}
We begin with the critical and sub-critical regimes $\beta \leq 1/2$ in \Cref{S3_diffusive_regime}, where we prove the convergence of the solution $v$ toward Gaussian and mixed self-similar profiles \emph{via} an entropy method.
Then, \Cref{S4_linear_regime} is devoted to the linear regime $\beta = 1$, giving the explicit expression of the solution and establishing the asymptotic behavior of the solution $v$.
Lastly, we address in \Cref{S5_super_critical_case} the super-critical regime $\beta > 1/2$, proving the convergence of the solution $v$ toward exponential self-similar profiles using Duhamel's principle and appropriate estimates on the arising convolution term.

\section{The sub-critical and critical regimes: \texorpdfstring{$0 \leq \beta \leq 1/2$}{0 <= beta <= 1/2}}\label{S3_diffusive_regime}

In this section, we focus on the diffusive scale $b(t)=\mathcal{O}(\sqrt{t})$. We prove that the solution converges at a polynomial rate to some self-similar profile, based on entropy dissipation arguments.
This corresponds to \hyperref[TH_asymptotic_behavior_solutions]{Theorem \ref*{TH_asymptotic_behavior_solutions}} \hyperref[TH_asymptotic_behavior_solutions]{\emph{(i)}} and \hyperref[TH_asymptotic_behavior_solutions]{\emph{(ii)}}
that are treated in \cref{SS33_beta_lower_1_2} and \cref{SS32_beta_1_2}, respectively.

\subsection{Self-similar rescaling for $0\leq \beta \leq 1/2$}

We make the change of variable
\begin{equation}\label{self sim 2}
	\tau:=\log\sqrt{1+t},
	\qquad
	y:=\frac{x}{\sqrt{1+t}},
	\qquad
	v(t,x)=
	\frac{1}{\sqrt{1+t}} \;
	w
	\Big(
	\log\sqrt{1+t},
	\frac{x}{\sqrt{1+t}}	\Big).
\end{equation}
Then the function $w$ satisfies the following problem
\begin{equation}\label{w(tau) beta<=1/2}
	\begin{aligned}
		\begin{cases}
				\partial_{\tau}w =
			2d \partial_{yy} w + \partial_{y} \big((y+\psi(\tau)) w\big),
			\qquad &
			\tau>0,\;y>0,\\[2mm]
			- 2d \partial_{y} w(\tau,0) = \psi(\tau) w(\tau,0) , \qquad &
			\tau>0, 
		\end{cases}
	\end{aligned}
\end{equation}
with $w_0=v_0$, where
    \begin{equation}
	\label{psi_beta<=1/2}
	\psi(\tau):= 2c\beta e^{\tau(2\beta-1)}.
    \end{equation}
 
We notice that 
 when $\beta=1/2$, we have $\psi(\tau)\equiv c$, and the associated stationary problem reads
 \begin{equation}\label{stationary pb beta=1/2}
		 \left\lbrace
		 \begin{array}{llll}
		 2d \partial_{yy} W + \partial_{y}\big((y + c)W\big) = 0, & \qquad & y > 0, & \\[1mm]
		 -2d \partial_{y}W(0) = cW(0). & &
		 \end{array}
		 \right .
 \end{equation}

When $\beta\in\intervallefo{0}{\frac{1}{2}}$, we find that  $\psi\equiv0$ if $\beta=0$, and $\psi(\tau)= 2c\beta e^{\tau(2\beta-1)}$
vanishes as $\tau\to\infty$ if $\beta<1/2$.
Therefore, the corresponding stationary problem writes in this case

\begin{equation}\label{EQ_sta_for_W_beta_lower_1/2}
	\left\lbrace
	\begin{array}{llll}
	2d \partial_{yy} W + \partial_{y}(yW) = 0, & \qquad & y > 0, & \\[2mm]
	\partial_{y}W(0) = 0. & &
	\end{array}
	\right .
\end{equation}
It then follows that for $y\ge 0$,
\begin{equation}
	\label{W-beta <= 1/2}
	W(y) =
	\begin{cases}
		W(0) \exp\prtHHH{\!\!-\dfrac{y^{2}}{4d} - \dfrac{cy}{2d}}, & \text{when } \beta = {1}/{2}, \\[4mm]
		W(0) \exp\prtHHH{\!\!-\dfrac{y^{2}}{4d}}, & \text{when } \beta \in [0, {1}/{2}),
	\end{cases}
\end{equation}
where in each case
$W(0)\geq 0$
is a uniquely determined constant such that the mass of the function $W$ equals $M$, i.e., $\int_{0}^{\infty}W(y)\md y=M$.

The main strategy in dealing with the diffusive regime is the entropy approach, together with the logarithmic Sobolev inequality \cite{GrossLogarithmic75, BakryAnalysis14} and the Csisz\'ar--Kullback inequality \cite{CsiszarInformationtype67, CoverElements06}, which we recall below.

\begin{lemma}[Logarithmic Sobolev inequality \cite{OttoGeneralization00}]\label{Lem_logSobolev}
Let $F(y)=\exp(-\Phi(y))$ be a smooth density on $\mathbb{R}_+$ and assume that there is a constant $\rho$ such that $\Phi''(y)\ge \rho>0$ for any $y>0$.
Then for any nonnegative $f\in L^{1}_{y}\prth{\mathbb{R}_{+}}$ satisfying $\int_{0}^{\infty}f(y)\md y=\int_{0}^{\infty}F(y)\md y$, we have
\begin{equation}\label{EQ_logSob} 
\int_0^\infty f(y)\log\bigg(\frac{f(y)}{F(y)}\bigg)\md y
\le \frac{1}{2\rho}\int_0^\infty f(y)\left(\partial_y\log\bigg(\frac{f(y)}{F(y)}\bigg)\right)^2\!\md y.
\end{equation}
\end{lemma}

\begin{lemma}[Csisz\'ar--Kullback inequality \cite{LepoutreCell14}]\label{Lem_Csiszar_Kullback} For any nonnegative functions $f,F\in L^1_{y}(\mathbb{R}_+)$ such that $\int_{0}^{\infty}f(y)\md y=\int_{0}^{\infty}F(y)\md y=M$, we have that
\begin{equation}\label{C-K ineq}
	\Vert f-F\Vert_{L^1_{y}(\mathbb{R}_{+})}^2\le 2M \int_0^\infty f(y)\log\bigg(\frac{f(y)}{F(y)}\bigg) \, \md y.
\end{equation}
\end{lemma}

\subsection{Long-time behavior for \texorpdfstring{$\beta=1/2$}{beta = 1/2}}\label{SS32_beta_1_2}

~
\vspace{-5mm}

\begin{proof}[\texorpdfstring{Proof of \hyperref[TH_asymptotic_behavior_solutions]{Theorem \ref*{TH_asymptotic_behavior_solutions} \emph{(ii)}}}{Proof of Theorem \ref*{TH_asymptotic_behavior_solutions} \emph{ii}}]
Since $\beta=1/2$, equation \eqref{w(tau) beta<=1/2} can be reduced to 
\begin{equation}\label{w-eqn:beta=1/2}
	\left\lbrace \begin{array}{llll}
		\partial_{\tau}w =
			2d \partial_{yy} w + \partial_{y} \big((y+c)w\big),
		& \qquad &
		\tau>0, & y>0,\\[2mm]
		- 2d \partial_{y} w(\tau,0) = c w(\tau,0) , & \qquad &
		\tau>0, & 
	\end{array} \right .
\end{equation}
	with $w_0=v_0$, for which the corresponding stationary solution $W$ is given in \eqref{W-beta <= 1/2}.
	The goal is to establish the convergence of the solution $w$ to problem \eqref{w-eqn:beta=1/2} toward $W$.
	
	By introducing the relative entropy associated with \eqref{w-eqn:beta=1/2} as
\begin{equation}\label{EQ_entropy}
		\mathcal{H}(w|W) : =\int_0^\infty w(\tau,y)\log\Prth{\frac{w(\tau,y)}{W(y)}} \md y \geq 0,
		\qquad \forall\tau\geq 0,
\end{equation}
	which is nonnegative and vanishes if and only if $w=W$ (due to Jensen's inequality),
	we find that
\pagebreak
\renewcommand{\gap}{2mm}
\begin{align}
			\frac{\md}{\md \tau}\mathcal{H}(w|W)&=\int_0^\infty \partial_{\tau}w(\tau,y)\Big(\log (w(\tau,y))+\frac{1}{4d}y^{2} + \frac{c}{2d}y\Big) \, \md y\label{EQ_need_mass_preservation}\\[\gap]
	&=\int_0^\infty\partial_y\Big(
	2dw_{y}(\tau,y)+(y+c) w(\tau,y)
	\Big) \Big(\log (w(\tau,y))+\frac{1}{4d}y^{2} + \frac{c}{2d}y\Big) \, \md y\nonumber\\[\gap]
	&=-\int_0^\infty\Big(
	2dw_{y}(\tau,y)+(y+c) w(\tau,y)
	\Big) \bigg(\frac{ \partial_{y}w(\tau,y)}{w(\tau,y)}+\frac{1}{2d}
 (y+c)
 \bigg) \, \md y\nonumber\\[\gap]
	&=-\int_0^\infty
	\Croch{2d\frac{ \partial_{y}w(\tau,y)^2}{w(\tau,y)}+2(y+c) \partial_{y}w(\tau,y)+\frac{1}{d}(y+c)^2w(\tau,y)} \, \md y\nonumber\\[\gap]
	&=-\int_0^\infty w(\tau,y)\bigg(
	   \sqrt{2d}\partial_y\log (w(\tau,y)) +\frac{1}{\sqrt{2d}}(y+c)\bigg)^2 \, \md y\nonumber\\[\gap]
			&=-2d \mathcal{I}(w|W),\label{1}
\end{align}
	where $\mathcal{I}(w|W)$ is called the Fisher information, given by
\begin{equation*}
		\mathcal{I}(w|W) : =\int_0^\infty w(\tau,y)\Prth{\partial_y\log \Prth{\frac{w(\tau,y)}{W(y)}}}^2 \md y
		\qquad \forall \tau\geq 0.
\end{equation*}

Notice that $W''(y)=\frac{1}{2d}$, then it follows from the Logarithmic Sobolev inequality 
that 
\begin{equation}\label{2}
		\mathcal{H}(w|W)\le d \mathcal{I}(w|W).
\end{equation}
Therefore, we infer from \eqref{1} and \eqref{2} that
\begin{align*}
		\frac{\md}{\md \tau}\mathcal{H}(w|W)+2\mathcal{H}(w|W)=-2d \mathcal{I}(w|W)+2\mathcal{H}(w|W)\le 0,
\end{align*}
which implies from Grönwall's lemma that
\begin{equation*}
		\mathcal{H}(w|W)\le \mathcal{H}(w_0|W)e^{-2\tau}.
\end{equation*}

Using the Csisz\'ar--Kullback inequality 
yields that
\begin{equation*}
		\big\Vert w(\tau,y)-W(y)\big\Vert_{L^1_y(\mathbb{R}_{+})}\le \sqrt{2M\mathcal{H}(w_0|W)}e^{-\tau},
		\qquad \forall\tau\geq 0.
\end{equation*}
Turning back to function $v$ using \eqref{self sim 2}, we finally obtain \eqref{beta=1/2: entropy conclu} with $\ell := \sqrt{2M\mathcal{H}(w_0|W)}$.
This concludes the proof of
\hyperref[TH_asymptotic_behavior_solutions]{Theorem~\ref*{TH_asymptotic_behavior_solutions}~\emph{(ii)}}.
\end{proof}

\subsection{Long-time behavior for \texorpdfstring{$0\leq \beta < 1/2$}{0 <= beta < 1/2}}\label{SS33_beta_lower_1_2}

~
\vspace{-5mm}

\begin{proof}[\texorpdfstring{Proof of \hyperref[TH_asymptotic_behavior_solutions]{Theorem \ref*{TH_asymptotic_behavior_solutions} \emph{(i)}}}{Proof of Theorem \ref*{TH_asymptotic_behavior_solutions} \emph{i}}]
	Similar to the case of $\beta=1/2$, again we introduce for any $\tau\geq 0$ the relative entropy:
\begin{equation*}
		\mathcal{H}(w|W) =\int_0^\infty w(\tau,y)\log\Prth{\frac{w(\tau,y)}{W(y)}} \md y,
\end{equation*}
and the Fisher information:
\begin{align*}
	\mathcal{I}(w|W)&=\int_0^\infty w(\tau,y)\Prth{\partial_y\log \Prth{\frac{w(\tau,y)}{W(y)}}}^2 \, \md y\\
	&=-\int_0^\infty
	\Croch{\frac{ \partial_{y}w(\tau,y)^2}{w(\tau,y)}+\frac{y}{d} \partial_{y}w(\tau,y)+\frac{y^2}{4d^2}w(\tau,y)} \md y.
\end{align*}
It follows that
\renewcommand{\gap}{2mm}
\renewcommand{\gapp}{-16mm}
\begin{align}
	\frac{\md}{\md \tau}\mathcal{H}(w|W)&=\int_0^\infty \partial_{\tau}w(\tau,y)\Big(\log (w(\tau,y))+\frac{1}{4d}y^{2}\Big) \, \md y\nonumber\\[\gap]
	&\hspace{\gapp}=\int_0^\infty\partial_{y}\Big(
	2dw_{y}(\tau,y)+(y+\psi(\tau)) w(\tau,y)
	\Big) \Big(\log (w(\tau,y))+\frac{1}{4d}y^{2}\Big) \, \md y\nonumber\\[\gap]
	&\hspace{\gapp}=-\int_0^\infty\Big(
	2dw_{y}(\tau,y)+(y+\psi(\tau)) w(\tau,y)
	\Big) \bigg(\frac{ \partial_{y}w(\tau,y)}{w(\tau,y)}+\frac{y}{2d}\bigg) \, \md y\nonumber\\[\gap]
	&\hspace{\gapp}=-\int_0^\infty
	\Croch{2d\frac{ \partial_{y}w(\tau,y)^2}{w(\tau,y)}+\big(2y+\psi(\tau)\big) \partial_{y}w(\tau,y)+\frac{y}{2d}\big(y+\psi(\tau)\big)w(\tau,y)} \md y\nonumber\\[\gap]
	&\hspace{\gapp}=-2d\int_0^\infty\left(
	\frac{ \partial_{y}w(\tau,y)^2}{w(\tau,y)}+\frac{y}{d} \partial_{y}w(\tau,y)+\frac{y^2}{4d^2}w(\tau,y)\right)+\psi(\tau)\left(\partial_{y}w(\tau,y)+\frac{y}{2d}w(\tau,y)\right)\md y\nonumber\\[\gap]
	&\hspace{\gapp}=-2d \mathcal{I}(w|W)+\psi(\tau)\left(w(\tau,0)-\frac{1}{2d}\int_0^\infty yw(\tau,y) \, \md y\right).\label{EQ_entropy_diff_yet_sharp}
\end{align}

\begin{remark}\label{REM_beta=0}
In the case $\beta=0$, \eqref{EQ_entropy_diff_yet_sharp} simply reduces to 
$ \frac{\md}{\md \tau}\mathcal{H}(w|W)= -2d \mathcal{I}(w|W)$ by noticing that $\psi(\tau)\equiv 0$. Following the lines in the case $\beta=1/2$, we obtain in this case that
\begin{equation*}
		\big\Vert w(\tau,y)-W(y)\big\Vert_{L^1_y(\mathbb{R}_{+})}\le \sqrt{2M\mathcal{H}(w_0|W)}e^{-\tau},
		\qquad \forall\tau\geq 0.
\end{equation*}
This produces the optimal convergence rate $\mathcal{O}(t^{-1/2})$ as stated in \Cref{PROPO_BETA_0} for the Neumann heat equation.
\end{remark}

Let us look at the case $\beta\in\intervalleoo{0}{\frac{1}{2}}$.
Since the precise behavior of the quantity $w(\tau,0)-\frac{1}{2d}\int_0^\infty y w(\tau,y)\md y$ in \eqref{EQ_entropy_diff_yet_sharp} remains unknown, we make a rough estimate by considering the worst-case scenario, relying on the nonnegativity of both $\psi(\tau)$ and the first moment of $w$. This leads to the inequality:
\begin{equation}\label{3}
	\frac{\md}{\md \tau}\mathcal{H}(w|W)\le -2d \mathcal{I}(w|W)+\psi(\tau)w(\tau,0).
\end{equation}
The Logarithmic Sobolev inequality 
leads to $\mathcal{H}(w|W)\le d\mathcal{I}(w|W)$ which together with \eqref{3} gives 
\begin{equation}\label{4}
	\frac{\md}{\md \tau}\mathcal{H}(w|W)
	\le -2\mathcal{H}(w|W)+\psi(\tau)w(\tau,0).
\end{equation}
Before proceeding, let us give a control of $w(\tau,0)$.
\begin{lemma}\label{lem5}
	Assume that $\beta\in\intervalleoo{0}{\frac{1}{2}}$ and that $w_0$ is bounded, compactly supported, and nonnegative in $\mathbb{R}_+$, then there exists $\Lambda>0$ such that $w(\tau,0)<\Lambda$ for every $\tau\ge 0$.
\end{lemma}

\begin{proof}[Proof of \Cref{lem5}]
	The proof follows from a comparison argument.
		Define 
	\begin{equation*}
		\overline w(\tau,y):=\exp\left(-\frac{y^2}{4d}-\frac{\psi(\tau)y}{2d}\right),
		\qquad
		\forall\tau\ge 0,\;\, \forall y\ge 0,
	\end{equation*}
then we can easily check that, for any $\tau>0$ and $y>0$,
\begin{align*}
		\partial_{\tau}w -
	2d \partial_{yy} w - \partial_{y} \big((y+\psi(\tau) )w\big)=-\frac{\psi'(\tau)}{2d}y\overline w(\tau,y)=\frac{2c\beta(1-2\beta)}{2d}e^{\tau(2\beta-1)}y\overline w\ge 0,
\end{align*}
and 
\begin{equation*}
	- 2d \partial_{y} w(\tau,0) - \psi(\tau) w(\tau,0)=0.
\end{equation*}

In addition, we can choose $\Lambda>0$ large such that $\Lambda\overline w_0(y)\ge w_0(y)=v_0(y)$ for $y\in\mathbb{R}_+$. We then conclude that $\Lambda\overline w$ is a supersolution to problem \eqref{w(tau) beta<=1/2} in $\intervallefo{0}{\infty}\times\mathbb{R}_{+}$ and the comparison principle implies that $0\le w(\tau,y)\le \Lambda\overline w(\tau,y)$ in $\intervallefo{0}{\infty}\times\mathbb{R}_{+}$. This implies in particular that $0\le w(\tau,0)\le \Lambda\overline w(\tau,0)=\Lambda$ for every $\tau\ge 0$. 
\end{proof}

\medskip

Now we turn back to \eqref{4} and obtain from \Cref{lem5} that 
\begin{equation*}
 		\frac{\md}{\md \tau}\mathcal{H}(w|W)\le -2\mathcal{H}(w|W)+\psi(\tau)\Lambda.
 \end{equation*}
This ordinary differential inequality can be solved by the method of variation of parameters:
\begin{equation*}
	\mathcal{H}(w|W)\le \mathcal{H}(v_0|W)e^{-2\tau}+\Lambda e^{-2\tau}\int_0^\tau \psi(s)e^{2s} \, \md s,
\end{equation*}
which together with \eqref{psi_beta<=1/2} implies that
\begin{align}
		\mathcal{H}(w|W)&\le \mathcal{H}(v_0|W)e^{-2\tau}+2c\beta\Lambda e^{-2\tau}\int_0^\tau e^{s(2\beta-1)} e^{2s} \, \md s\nonumber\\
		&\le \mathcal{H}(v_0|W)e^{-2\tau}+\frac{2c\beta\Lambda }{2\beta+1}\left(e^{(2\beta-1)\tau}-e^{-2\tau}\right)\nonumber\\
		&=\left(\mathcal{H}(v_0|W)-\frac{2c\beta\Lambda }{2\beta+1}\right)e^{-2\tau}+\frac{2c\beta\Lambda }{2\beta+1}e^{(2\beta-1)\tau}\nonumber\\
		&\le k e^{(2\beta-1)\tau},\nonumber
\end{align}
with $k:=\mathcal{H}(v_0|W)+\frac{2c\beta\Lambda }{2\beta+1}>0$.

\smallskip

It then follows from the Csisz\'ar--Kullback inequality 
that 
\begin{equation*}
		\big\Vert w(\tau,y)-W(y)\big\Vert_{L^1_y(\mathbb{R}_{+})}\le \sqrt{2Mk}e^{(\beta-\frac{1}{2})\tau},
		\qquad \forall\tau>0.
\end{equation*}
Therefore, turning back to the function $v$ using \eqref{self sim 2}, we obtain \eqref{beta<1/2: entropy + Duhamel} with $\ell : = \sqrt{2Mk}$. This completes the proof of
\hyperref[TH_asymptotic_behavior_solutions]{Theorem~\ref*{TH_asymptotic_behavior_solutions}~\emph{(i)}}.
\end{proof}

\section{The linear regime: \texorpdfstring{$\beta=1$}{beta = 1}}\label{S4_linear_regime}

We consider here the linear case $\beta=1$, corresponding to the moving boundary $b\prth{t}=ct$. In this regime, equation \eqref{v-eqn}$|_{b(t)=ct}$ simplifies to
\begin{equation}\label{EQ_linear_problem_on_v}
\left\lbrace \begin{array}{lllll}
	\partial_{t}v = d \partial_{xx}v+c\partial_x v, & \qquad &
	t>0, & x>0,\\[0.9mm]
	-d\partial_{x}v = cv , & \qquad &
	t>0, & x=0.
\end{array} \right .
\end{equation}
This section is dedicated to the proof of \Cref{TH_fundamental_sol_for_beta_1} which splits in two parts.
The first concerns the explicit construction of the solution to \eqref{EQ_linear_problem_on_v}, relying on suitable extension arguments. The second establishes the asymptotic $L^{1}$-convergence of $v\prth{t,\point}$ toward the exponential stationary profile $V$, based on a detailed analysis of the fundamental solution.
These two aspects are treated in \cref{SS41_fundamental_sol} and \cref{SS42_asymptotic_behavior}, respectively.

\subsection{Fundamental solution for $\beta=1$}\label{SS41_fundamental_sol}

~
\vspace{-5mm}

\begin{proof}[\texorpdfstring{Proof of \hyperref[TH_fundamental_sol_for_beta_1]{Theorem \ref*{TH_fundamental_sol_for_beta_1} \emph{(i)}}}{Proof of Theorem \ref*{TH_fundamental_sol_for_beta_1} \emph{i}}]
We anticipate an extension $\widetilde{v}_{0}$ of $v_{0}$ in $\acco{x<0}$, such that the solution $\widetilde{v}$ to
\begin{equation}\label{EQ_tilde_v}
\left\lbrace \begin{array}{llll}
	\partial_{t}\widetilde{v} = d\partial_{xx} \widetilde{v}+c\partial_{x} \widetilde{v}, & \qquad &
	t>0, & x\in\mathbb{R},\\[1mm]
	\widetilde{v}|_{t=0} = \widetilde{v}_{0} , & \qquad &
	t=0, & x\in \mathbb{R},
\end{array} \right .
\end{equation}
satisfies $\widetilde{v}(t,x)|_{x>0}= v(t,x)$.
To achieve this, we need to check that the flux condition
$-d\partial_{x}\widetilde{v} = c \widetilde{v}$
is verified
at $x=0$ for any positive time\footnote{Checking this condition for all positive times is essential, as it is precisely what allows the extension argument to succeed when $\beta=1$---in which case equation \eqref{v-eqn} is autonomous.}---this recovers the Robin boundary condition, namely second line in \eqref{EQ_linear_problem_on_v}.
Explaining the solution $\widetilde{v}$ to \eqref{EQ_tilde_v} by convolving $v_{0}$ with the biased heat kernel $G\prth{t,x+ct}$, this condition reads
$$
\int_{\xi=0}^{\infty}
	G\prth{t,ct+\xi}\crocHH{c \widetilde{v}_{0}\prth{-\xi} + d \partial_{x} \widetilde{v}_{0}\prth{-\xi} }
	+
	G\prth{t,ct-\xi}\crocHH{c v_{0}\prth{\xi} + d \partial_{x} v_{0}\prth{\xi} }
\, \md \xi = 0,
$$
and can be recast, thanks to the algebraic trick
$G\prth{t,ct+\xi} = G\prth{t,ct-\xi}e^{-\frac{c}{d}\xi}$, as
$$
\int_{\xi=0}^{\infty}
	G\prth{t,ct-\xi}
	\underbrace{\crocHH{
		e^{-\frac{c}{d}\xi} \prtH{c \widetilde{v}_{0}\prth{-\xi} + d \partial_{x} \widetilde{v}_{0}\prth{-\xi}}
		+
		\prtH{c v_{0}\prth{\xi} + d \partial_{x} v_{0}\prth{\xi}}
	}}_{\text{We ask then the vanishing of this quantity for all $\xi>0$.}}
\md \xi = 0.
$$
Therefore, for $x<0$, $\widetilde{v}_{0}$ should solve the following linear ODE:
\begin{equation}\label{EQ_ode_tilde_v0}
\left\lbrace \begin{array}{lll}
	d\partial_{x} \widetilde{v}_{0}\prth{x} + c \widetilde{v}_{0}\prth{x} = - e^{-\frac{c}{d}x} \crocHH{d\partial_{x}v_{0}\prth{-x} + c v_{0}\prth{-x}}, & \qquad & x<0,\\[1.3mm]
	\widetilde{v}_{0}\prth{0} = v_{0}\prth{0} , & \qquad & x=0,
\end{array} \right .
\end{equation}
where the second line ensures the compatibility of $\widetilde{v}_{0}$ regarding the Robin boundary condition at $x=0$.
As a result, the extended initial data $\widetilde{v}_{0}$ is given by
\begin{equation}\label{EQ_tilde_v0}
	\widetilde{v}_{0}\prth{x} =
	\begin{cases}
	v_{0}\prth{x}, & \quad \text{if }x\geq 0,\\[1.3mm]
	e^{-\frac{c}{d}x}\crocHH{v_{0}\prth{-x} + \frac{c}{d}\int_{0}^{-x}v_{0}\prth{\xi} \, \md \xi}, & \quad \text{if }x\leq 0.
	\end{cases}
\end{equation}
We can now provide a first expression for $v$:
\begin{align}
	v\prth{t,x} &=
	\int_{\xi=0}^{\infty}
		G\prth{t,x+ct-\xi}v_{0}\prth{\xi}
	\, \md \xi +
	\int_{\xi=0}^{\infty}
		G\prth{t,x+ct+\xi}e^{\frac{c}{d}\xi}v_{0}\prth{\xi} 
	\, \md \xi \label{EQ_almost_kernel_v_1}\\[3mm]
	& \hspace{43mm} + \frac{c}{d} \int_{\xi=0}^{\infty}
		G\prth{t,x+ct+\xi} e^{\frac{c}{d}\xi}
		\int_{\omega=0}^{\xi}
			v_{0}\prth{\omega}
		\md \omega
	\, \md \xi, \label{EQ_almost_kernel_v_2}
\end{align}
which can be recast into \eqref{EQ_fundamental_sol_v_linear} with the kernel $H$ given by \eqref{EQ_heat_kernel_linear} by using Fubini's theorem on \eqref{EQ_almost_kernel_v_2}. This concludes the proof of \hyperref[TH_fundamental_sol_for_beta_1]{Theorem \ref*{TH_fundamental_sol_for_beta_1} \emph{(i)}}.
\end{proof}

\subsection{Long-time behavior for $\beta=1$}\label{SS42_asymptotic_behavior}

~
\vspace{-5mm}

\begin{proof}[\texorpdfstring{Proof of \hyperref[TH_fundamental_sol_for_beta_1]{Theorem \ref*{TH_fundamental_sol_for_beta_1} \emph{(ii)}}}{Proof of Theorem \ref*{TH_fundamental_sol_for_beta_1} \emph{ii}}]
	This proof is divided into two steps. We show first that there exist some positive constants $\tilde{\ell}$, $\tilde{k}$ and $t_{0}$ such that
	\begin{equation}\label{EQ_cv_to_V_when_beta_equals_1_first}
		\vertii{
			v\prth{t,x} - V\prth{x}
		}{L^{1}_{x}\prth{\mathbb{R}_{+}^{}}}
		\leq
		\tilde{\ell} e^{-\tilde{k} t},
		\qquad
		\forall t>t_{0}.
	\end{equation}
	We prove then the uniform boundedness of $v$ which enables, combined with \eqref{EQ_cv_to_V_when_beta_equals_1_first}, to reach \eqref{EQ_cv_to_V_when_beta_equals_1} for any $t>0$.

	\medskip

	As a stationary solution to \eqref{EQ_linear_problem_on_v}, the function $V$, defined in \eqref{EQ_def_V_beta_equals_1}, can be rewritten using \eqref{EQ_almost_kernel_v_1}-\eqref{EQ_almost_kernel_v_2} as
	\begin{align*}
		V\prth{x} &=
		\int_{\xi=0}^{\infty}
			G\prth{t,x+ct-\xi}V\prth{\xi}
		\, \md \xi +
		\int_{\xi=0}^{\infty}
			G\prth{t,x+ct+\xi}e^{\frac{c}{d}\xi}V\prth{\xi} 
		\, \md \xi\\[3mm]
		& \hspace{43mm} + \frac{c}{d} \int_{\xi=0}^{\infty}
			G\prth{t,x+ct+\xi} e^{\frac{c}{d}\xi}
			\int_{\omega=0}^{\xi}
				V\prth{\omega}
			\md \omega
		\, \md \xi,
	\end{align*}
	for any $t>0$.
	This equality, combined with \eqref{EQ_almost_kernel_v_1}-\eqref{EQ_almost_kernel_v_2}, allows us to write $v\prth{t,x}-V\prth{x}$ as
	\begin{align*}
		v\prth{t,x}-V\prth{x}
		&=
		\int_{\xi=0}^{\infty}
			G\prth{t,x+ct-\xi}\prth{v_{0}\prth{\xi}-V\prth{\xi}}
		\, \md \xi\\[2mm]
		& \hspace{11mm} + \int_{\xi=0}^{\infty}
			G\prth{t,x+ct+\xi}e^{\frac{c}{d}\xi}\prth{v_{0}\prth{\xi}-V\prth{\xi}} 
		\, \md \xi\\[2mm]
		& \hspace{22mm} - \frac{c}{d} \int_{\xi=0}^{\infty}
			G\prth{t,x+ct+\xi} e^{\frac{c}{d}\xi}
			\int_{\omega=\xi}^{\infty}
				\prth{v_{0}\prth{\omega}-V\prth{\omega}}
			\md \omega
		\, \md \xi.
	\end{align*}
	Notice, in the last integral above, we replaced
	$$
	\int_{\omega=0}^{\xi} \prth{v_{0}-V}
	\qquad
	\text{by}
	\qquad
	-\int_{\omega=\xi}^{\infty} \prth{v_{0}-V},
	$$
	owing to the fact that $\int_{0}^{\infty} v_{0} = \int_{0}^{\infty} V$. This is the only point of the proof where we use this assumption. Taking now the $L^{1}$-norm of $v\prth{t,\point}-V$, we have then
	\begin{align*}
		\vertii{v\prth{t,x}-V\prth{x}}{L^{1}_{x}\prth{\mathbb{R}_{+}^{}}}
		&\leq 
		\underbrace{\int_{\xi=0}^{\infty}
			\verti{v_{0}\prth{\xi}-V\prth{\xi}}
		\int_{x=ct}^{\infty}
			G\prth{t,x -\xi}
		\md x \, \md \xi}_{=:I_{1}\prth{t}}\\[2mm]
		& \hspace{5mm} + \underbrace{\int_{\xi=0}^{\infty}
			\verti{v_{0}\prth{\xi}-V\prth{\xi}} 
		e^{\frac{c}{d}\xi}
		\int_{x=ct}^{\infty}
			G\prth{t,x +\xi}
		\md x \, \md \xi}_{=:I_{2}\prth{t}}\\[2mm]
		& \hspace{10mm} + \frac{c}{d} \underbrace{\int_{\xi=0}^{\infty}
			\int_{\omega=\xi}^{\infty}
				\verti{v_{0}\prth{\omega}-V\prth{\omega}}
			\md \omega
		\, e^{\frac{c}{d}\xi}
		\int_{x=ct}^{\infty}
			G\prth{t,x +\xi}
		\md x \, \md \xi,}_{=:I_{3}\prth{t}}
	\end{align*}
	where we have performed the change of variable $\bar{x} = x+ct$, and then renamed $\bar{x}$ back to $x$, and applied Fubini's theorem to interchange the integration orders.

	Our objective is now to bound $I_{1}$, $I_{2}$ and $I_{3}$ from above. Before proceeding, let us state the following lemma which we shall rely on in the next steps.

	\begin{lemma}\label{LE_beta_equals_1}
		For any $t>0$, there holds
		\begin{equation}\label{EQ_lemma_beta_equals_1_a}
			\phantom{e^{\frac{c}{d}\xi}}\int_{x=ct}^{\infty}G\prth{t,x -\xi}\, \md x
			\leq 
			\exp \crocHHH{
				-\Prth{\frac{c}{2\sqrt{d}}\sqrt{t} - \frac{\xi}{2\sqrt{dt}}}^{2}
			},
			\qquad \forall \xi \in \intervalleoo{0}{ct},
		\end{equation}
		\begin{equation}\label{EQ_lemma_beta_equals_1_b}
			e^{\frac{c}{d}\xi}
			\int_{x=ct}^{\infty}G\prth{t,x +\xi}\, \md x
			\leq 
			\exp \crocHHH{
				-\Prth{\frac{c}{2\sqrt{d}}\sqrt{t} - \frac{\xi}{2\sqrt{dt}}}^{2}
			},
			\qquad \forall \xi \in \intervalleoo{0}{\infty}.
		\end{equation}
	\end{lemma}
	
	\begin{proof}[Proof of \Cref{LE_beta_equals_1}]
		Expanding the integral in the left-hand side of \eqref{EQ_lemma_beta_equals_1_a}, we have, for any $\xi \in \intervalleoo{0}{ct}$,
		\begin{align*}
			\int_{x=ct}^{\infty} G\prth{t,x-\xi} \, \md x
			&=\frac{1}{\sqrt{4\pi dt}} \int_{x=ct}^{\infty}
			\exp \crocHHH{
				-\Prth{\frac{\prth{x - \xi}^{2}}{4dt}}
			} \, \md x\\[2mm]
			&=\frac{1}{2} \, \text{erfc}\Prth{\frac{c}{2\sqrt{d}}\sqrt{t}-\frac{\xi}{2\sqrt{dt}}},
		\end{align*}
		where we performed the change of variable
        $\bar{x} = \prth{x-\xi}/\prth{2\sqrt{dt}}$, and then renamed $\bar{x}$ back to $x$, to transition from the first line to the second.
		Now, since $\frac{c}{2\sqrt{d}}\sqrt{t}-\frac{\xi}{2\sqrt{dt}}$ is positive (as $0 < \xi < ct$), and using the fact that $\frac{1}{2}\text{erfc}\prth{\gamma} \leq \exp\prth{-\gamma^{2}}$ for any $\gamma>0$, we finally reach \eqref{EQ_lemma_beta_equals_1_a}.

		\medskip
		
		Similarly to the proof of \eqref{EQ_lemma_beta_equals_1_a}, we have
		\begin{align*}
			e^{\frac{c}{d}\xi}\int_{x=ct}^{\infty} G\prth{t,x+\xi} \, \md x
			&= \frac{e^{\frac{c}{d}\xi}}{\sqrt{4\pi dt}} \int_{x=ct}^{\infty}
			\exp \crocHHH{
				-\Prth{\frac{\prth{x + \xi}^{2}}{4dt}}
			} \, \md x\\[2mm]
			&=\frac{e^{\frac{c}{d}\xi}}{2} \, \text{erfc}\Prth{\frac{c}{2\sqrt{d}}\sqrt{t}+\frac{\xi}{2\sqrt{dt}}}\\[2mm]
			&\leq e^{\frac{c}{d}\xi} \,
			\exp \crocHHH{
				-\Prth{\frac{c}{2\sqrt{d}}\sqrt{t} + \frac{\xi}{2\sqrt{dt}}}^{2}
			}\\[2mm]
			& =\exp \crocHHH{
				-\Prth{\frac{c}{2\sqrt{d}}\sqrt{t} - \frac{\xi}{2\sqrt{dt}}}^{2}
			},
		\end{align*}
		which establishes \eqref{EQ_lemma_beta_equals_1_b}.
	\end{proof}

	\medskip

	We are now ready to begin estimating $I_1$, $I_2$, and $I_3$.

	\medskip

	$\bullet$ \emph{Control of $I_{1}\prth{t}$.} Splitting the integral $\int_{\xi=0}^{\infty}$ into $\int_{\xi=0}^{\frac{c}{2}t}+\int_{\xi=\frac{c}{2}t}^{\infty}$ in the expression of $I_{1}$, we have
	\begin{align}
		I_{1}\prth{t}
		&=
		\int_{\xi=0}^{\frac{c}{2}t}
			\verti{v_{0}\prth{\xi}-V\prth{\xi}}
		\int_{x=ct}^{\infty}
			G\prth{t,x -\xi}
		\md x \, \md \xi \label{EQ_control_I1_a}\\[2mm]
		&\hspace{30mm}
		+ \int_{\xi=\frac{c}{2}t}^{\infty}
			\verti{v_{0}\prth{\xi}-V\prth{\xi}}
		\int_{x=ct}^{\infty}
			G\prth{t,x -\xi}
		\md x \, \md \xi. \label{EQ_control_I1_b}
	\end{align}
	Now substituting \eqref{EQ_lemma_beta_equals_1_a} into \eqref{EQ_control_I1_a}, and noticing that
	$\int_{x=ct}^{\infty}G\prth{t,x-\xi}\md x\leq 1$ in \eqref{EQ_control_I1_b} yields
	\renewcommand{\gap}{1}
	\begin{align}
		\scalebox{\gap}{$\displaystyle I_{1}\prth{t}~$} & \scalebox{\gap}{$\displaystyle \leq
		\int_{\xi=0}^{\frac{c}{2}t}
			\verti{v_{0}\prth{\xi}-V\prth{\xi}}
		\times \exp \crocHHH{
				-\Prth{\frac{c}{2\sqrt{d}}\sqrt{t} - \frac{\xi}{2\sqrt{dt}}}^{2}
			} \, \md \xi 
		+ \int_{\xi=\frac{c}{2}t}^{\infty}
			\verti{v_{0}\prth{\xi}-V\prth{\xi}}
		\, \md \xi\label{EQ_reuse_I1} $}\\[2mm]
		&\scalebox{\gap}{$\displaystyle \leq 
		\exp \crocHHH{
				-\frac{c^{2}}{16d}t
			}
		\int_{\xi=0}^{\frac{c}{2}t}
			\verti{v_{0}\prth{\xi}-V\prth{\xi}}
		\, \md \xi
		+ \int_{\xi=\frac{c}{2}t}^{\infty}
			\verti{v_{0}\prth{\xi}-V\prth{\xi}} \, \md \xi. $}\nonumber
	\end{align}
	Remember at this point that $v_{0}$ is compactly supported. Hence, for $t_{0}>0$ chosen sufficiently large so that $\text{supp}\prth{v_{0}}\cap\{\xi>\frac{c}{2}t_{0}\} = \varnothing$, we have, for any $t>t_{0}$,
	\begin{equation*}\label{EQ_only_V_counts}
		\verti{v_{0}\prth{\xi}-V\prth{\xi}} = \verti{V\prth{\xi}}
		=\verti{\int_{\omega=0}^{\infty}v_{0}\prth{\omega}\md \omega} \frac{c}{d} \, e^{-\frac{c}{d}\xi},
		\qquad
		\forall \xi > \frac{c}{2}t.
	\end{equation*}
	As a consequence,
	\begin{align}
		I_{1}\prth{t} & \leq
		\vertii{v_{0}-V}{L\prth{\mathbb{R}_{+}^{}}}
		\times
		e^{-\frac{c^{2}}{16d}t} +
		\vertii{v_{0}}{L^{1}}
		\times
		\int_{\xi = \frac{c}{2}t}^{\infty}
			\frac{c}{d} \, e^{-\frac{c}{d}\xi}
		\, \md \xi\nonumber\\[2mm]
		&= 
		\vertii{v_{0}-V}{L\prth{\mathbb{R}_{+}^{}}}
		\times
		e^{-\frac{c^{2}}{16d}t} +
		\vertii{v_{0}}{L^{1}}
		\times
		e^{-\frac{c^{2}}{2d}t}\nonumber\\[2mm]
		&\leq
		\tilde{\ell}_{1}e^{-\tilde{k}_{1}t},\label{EQ_upper_bound_beta_equals_1_a}
	\end{align}
	for some positive constants $\tilde{\ell}_{1}$ and $\tilde{k}_{1}$.

	\medskip

	$\bullet$ \emph{Control of $I_{2}\prth{t}$.} From \eqref{EQ_lemma_beta_equals_1_b} in \Cref{LE_beta_equals_1}, we have
	\renewcommand{\gap}{1}
	\begin{align}
		\scalebox{\gap}{$\displaystyle I_{2}\prth{t} ~$}
		&\scalebox{\gap}{$\displaystyle \leq \int_{\xi=0}^{\infty}
			\verti{v_{0}\prth{\xi}-V\prth{\xi}}
			\times \exp \crocHHH{
				-\Prth{\frac{c}{2\sqrt{d}}\sqrt{t} - \frac{\xi}{2\sqrt{dt}}}^{2}
			} \, \md \xi, $}\nonumber\\
		&\scalebox{\gap}{$\displaystyle \leq
		\int_{\xi=0}^{\frac{c}{2}t}
			\verti{v_{0}\prth{\xi}-V\prth{\xi}}
		\times \exp \crocHHH{
				-\Prth{\frac{c}{2\sqrt{d}}\sqrt{t} - \frac{\xi}{2\sqrt{dt}}}^{2}
			} \, \md \xi 
		+ \int_{\xi=\frac{c}{2}t}^{\infty}
			\verti{v_{0}\prth{\xi}-V\prth{\xi}}
		\, \md \xi $},\nonumber
	\end{align}
	which is precisely the right-hand-side of \eqref{EQ_reuse_I1}. Following then the same computation that allowed to reach \eqref{EQ_upper_bound_beta_equals_1_a} in the control of $I_{1}$, we get
	\begin{equation}\label{EQ_upper_bound_beta_equals_1_b}
	I_{2}\prth{t}\leq \tilde{\ell}_{2}e^{-\tilde{k}_{2}t},
	\end{equation}
	with $\prth{\tilde{\ell}_{2} , \tilde{k}_{2}} = \prth{\tilde{\ell}_{1} , \tilde{k}_{1}}$.

	\medskip

	$\bullet$ \emph{Control of $I_{3}\prth{t}$.} Similar arguments as in the controls of $I_{1}$ and $I_{2}$ yield
	\begin{align}
		I_{3}\prth{t} 
		&= \int_{\xi=0}^{\infty}
		\int_{\omega=\xi}^{\infty}
				\verti{v_{0}\prth{\omega}-V\prth{\omega}}
			\md \omega
		\, e^{\frac{c}{d}\xi}
		\int_{x=ct}^{\infty}
			G\prth{t,x +\xi}
		\, \md x \, \md \xi\nonumber\\[2mm]
		&\leq \int_{\xi=0}^{\infty}
		\int_{\omega=\xi}^{\infty}
				\verti{v_{0}\prth{\omega}-V\prth{\omega}}
			\md \omega
		\times \exp \crocHHH{
				-\Prth{\frac{c}{2\sqrt{d}}\sqrt{t} - \frac{\xi}{2\sqrt{dt}}}^{2}
			}
		\, \md \xi\nonumber\\[2mm]
		&\leq
		e^{-\frac{c^{2}}{16d}t}
		\int_{\xi=0}^{\frac{c}{2}t} \int_{\omega = \xi}^{\infty}
			\verti{v_{0}\prth{\omega}-V\prth{\omega}}
		\md \omega\,\md \xi
		+
		\vertii{v_{0}}{L^{1}}
		\int_{\xi=\frac{c}{2}t}^{\infty} \int_{\omega = \xi}^{\infty}
			\frac{c}{d}e^{-\frac{c}{d}\omega}
		\, \md \omega \, \md \xi\nonumber\\[2mm]
		&\leq
		e^{-\frac{c^{2}}{16d}t}
		\int_{\xi=0}^{\infty} \int_{\omega = \xi}^{\infty}
			\verti{v_{0}\prth{\omega}-V\prth{\omega}}
		\md \omega\,\md \xi
		+
		\frac{d}{c}
		\vertii{v_{0}}{L^{1}}
		e^{-\frac{c^{2}}{2d}t}\nonumber\\[2mm]
		&=
		e^{-\frac{c^{2}}{16d}t}
		\int_{\omega=0}^{\infty}
		\omega \verti{v_{0}\prth{\omega}-V\prth{\omega}} \, \md \omega
		+
		\frac{d}{c}
		\vertii{v_{0}}{L^{1}}
		e^{-\frac{c^{2}}{2d}t},\label{EQ_moment}
	\end{align}
	where we applied Fubini's theorem to derive the last line. It remains to observe that the integral in \eqref{EQ_moment} is bounded, since both $\verti{v_{0}}$ and $\verti{V}$ have finite first moment. This results in
	\begin{equation}\label{EQ_upper_bound_beta_equals_1_c}
		I_{3}\prth{t} \leq \tilde{\ell}_{1}e^{-\tilde{k}_{1}t},
	\end{equation}
	for some positive constants $\tilde{\ell}_{3}$ and $\tilde{k}_{3}$.

	By gathering
	\eqref{EQ_upper_bound_beta_equals_1_a}, \eqref{EQ_upper_bound_beta_equals_1_b} and \eqref{EQ_upper_bound_beta_equals_1_c}, we can eventually find some positive constants $\tilde{\ell}$, $\tilde{k}$ and $t_{0}$ so that \eqref{EQ_cv_to_V_when_beta_equals_1_first} holds for any $t>t_{0}$.

	\medskip

	We now focus on the boundedness of the solution $v$ to extend \eqref{EQ_cv_to_V_when_beta_equals_1_first} for any $t>0$ (up to change $\prth{\tilde{\ell},\tilde{k}}$ into $\prth{{\ell},{k}}$).

	Owing to the comparison principle, there are two real constants $\underline{\lambda}\leq \overline{\lambda}$ such that $v$ is sandwiched between the two stationary solutions $\underline{\lambda} \, \frac{c}{d}e^{-\frac{c}{d}x}$ and $\overline{\lambda} \, \frac{c}{d}e^{-\frac{c}{d}x}$, namely
	$$
	\underline{\lambda} \, \frac{c}{d}e^{-\frac{c}{d}x} \leq  v\prth{t,x} \leq \overline{\lambda} \, \frac{c}{d}e^{-\frac{c}{d}x},
	\qquad
	\forall t>0, \, \forall x>0.
	$$
	Hence,
	$$
	\prth{\underline{\lambda}-M} \, \frac{c}{d}e^{-\frac{c}{d}x} \leq  v\prth{t,x} -V\prth{x} \leq \prth{\overline{\lambda}-M} \, \frac{c}{d}e^{-\frac{c}{d}x},
	\qquad
	\forall t>0, \, \forall x>0,
	$$
	and therefore, for $\Lambda : = \max\prth{{\vert{\underline{\lambda}-M}\vert , \vert{\overline{\lambda}-M}\vert}}$,
	\begin{equation}\label{EQ_ready_to_integrate}
		\verti{v\prth{t,x}-V\prth{x}} \leq \Lambda \, \frac{c}{d}e^{-\frac{c}{d}x}
		\qquad
		\forall t>0, \, \forall x>0.
	\end{equation}
	Integrating the inequality \eqref{EQ_ready_to_integrate} over $\mathbb{R}_{+}$ then yields
	\begin{equation}\label{EQ_cv_to_V_when_beta_equals_1_second}
		\vertii{v\prth{t,x}}{L^{1}_{x}\prth{\mathbb{R}_{+}^{}}}
		\leq \Lambda,
		\qquad\forall t>0.
	\end{equation}

	\medskip

	Finally, we only need to combine \eqref{EQ_cv_to_V_when_beta_equals_1_first} and \eqref{EQ_cv_to_V_when_beta_equals_1_second} to eventually establish \eqref{EQ_cv_to_V_when_beta_equals_1} for any $t>0$. This concludes the proof of \hyperref[TH_fundamental_sol_for_beta_1]{Theorem~\ref*{TH_fundamental_sol_for_beta_1}~\emph{(ii)}}.
\end{proof}

\section{The super-critical regime: \texorpdfstring{$1/2<\beta\leq 1$}{1/2 < beta <= 1}}\label{S5_super_critical_case}

In this last section, we address the super-critical regime $b(t)\sim ct^{\beta}$ with $\beta > 1/2$.
We prove that the solution converges at a polynomial rate to some self-similar profile, based on Duhamel's principle and refined kernel estimates. This corresponds to point \hyperref[TH_asymptotic_behavior_solutions]{\emph{(iii)}} of \hyperref[TH_asymptotic_behavior_solutions]{Theorem \ref*{TH_asymptotic_behavior_solutions}}.

\subsection{Self-similar rescaling for $\beta > 1/2$}\label{SS21_self_similar}

We make the change of variable 
\begin{equation}\label{EQ_rescaling_super_sqrt}
	v\prth{t,x} : =
	\frac{b'\prth{t}}{c\beta} \,
	w \prtHHH{\,
		\overbrace{\int_{s=0}^{t} \crocHHH{\frac{b'\prth{s}}{c\beta}}^{2} ds}^{=:\tau}, \;
		\overbrace{\frac{b'\prth{t}}{c\beta}x\vphantom{\int_{0}^{t} [\frac{b'\prth{s}}{c\beta}]^2 \,ds}}^{=:y}
	\,},
\end{equation}
where we recall that the function $b$ is defined in \eqref{EQ_def_algebraic_b}, and that the function $v$ satisfies the equations given in \eqref{v-eqn}.
The relation between the variables $t$ and $\tau$ writes then
\begin{equation}\label{EQ_rel_between_tau_and_t}
	\tau = 
	\frac{1}{2\beta-1} \crocHH{\prth{1+t}^{2\beta-1} - 1}
	\quad \iff \quad
	t = 
	\prtHH{1+\prth{2\beta-1}\tau}^{\frac{1}{2\beta-1}} - 1,
\end{equation}
and the new unknown $w$ satisfies the following problem
\begin{equation}\label{EQ_for_w_beta_larger_1/2}
	\left\lbrace \begin{array}{llll}
		\partial_{\tau}w =
		d \partial_{yy} w + \partial_{y} \crocHH{ \prth{c\beta-\eta\prth{\tau}y}w},
		& \qquad &
		\tau>0, & y>0,\\[2mm]
		- d \partial_{y} w = c \beta w , & \qquad &
		\tau>0, & y=0,\\[2mm]
		w|_{\tau=0} = v_{0} , & \qquad &
		\tau=0, & y>0,
	\end{array} \right .
\end{equation}
where
\begin{equation}\label{EQ_eta}
	\eta\prth{\tau}
	:=
	\frac{\beta-1}{1 + \prth{2\beta-1}\tau}.
\end{equation}
Notice that $\eta\prth{\tau}$ has the sign of $\beta-1$ and vanishes as $\tau\to\infty$ in the regime $\beta>1/2$.
As a result, the stationary problem associated to \eqref{EQ_for_w_beta_larger_1/2} writes
\begin{equation}\label{EQ_sta_for_W_beta_larger_1/2}
	\left\lbrace \begin{array}{lll}
		\partial_{yy} W = - \frac{c\beta}{d} \partial_{y} W, & \qquad &
		y>0,\\[2mm]
		- d \partial_{y} W|_{y=0} = c\beta W|_{y=0} , & \qquad &
		y=0,
	\end{array} \right .
\end{equation}
which gives for $y>0$,
\begin{equation}\label{EQ_sta_for_W_beta_greater_1/2}
W\prth{y} = W\prth{0} \exp\Prth{-\frac{c\beta}{d}y},
\end{equation}
where
$W\prth{0}$
is uniquely determined so that the total masses of $v_{0}$ and $W$ are equal, that is $\int_{0}^{\infty}W(y)\md y=M$.

\subsection{Long-time behavior for \texorpdfstring{$1/2 < \beta \leq 1$}{1/2 < beta <= 1}}\label{SS22_long_time_behavior}

~
\vspace{-5mm}

\begin{proof}[\texorpdfstring{Proof of \hyperref[TH_asymptotic_behavior_solutions]{Theorem \ref*{TH_asymptotic_behavior_solutions} \emph{(iii)}}}{Proof of Theorem \ref*{TH_asymptotic_behavior_solutions} \emph{iii}}]
	Observe in the first line of \eqref{EQ_for_w_beta_larger_1/2}, that $\eta\prth{\tau} \partial_{y}\prth{yw}$ is a vanishing perturbation as $\tau$ goes to $+\infty$.
	Hence, it is natural to consider this quantity as an additional source term to the following evolution problem
	\begin{equation}\label{EQ_for_w_beta_larger_1/2_linear}
		\left\lbrace \begin{array}{llll}
			\partial_{\tau}\tilde{w} =
			d \partial_{yy} \tilde{w} + c\beta \partial_{y}\tilde{w},
			& \qquad &
			\tau>0, & y>0,\\[2mm]
			- d \partial_{y} \tilde{w}|_{y=0} = c\beta \tilde{w}|_{y=0} , & \qquad &
			\tau>0, & y=0,\\[2.7mm]
			\tilde{w}|_{\tau=0} = v_{0} , & \qquad &
			\tau=0, & y>0,
		\end{array} \right .
	\end{equation}
	which has the same form as the linear case \eqref{EQ_linear_problem_on_v}, with the constant $c$ replaced by $c\beta$ now.
    Therefore, it follows from \Cref{TH_fundamental_sol_for_beta_1} that there are some positive constants $\ell$ and $k$ such that
	\begin{equation*}
		\bigg\Vert
			\tilde{w}\prth{\tau,y} - \underbrace{\frac{c\beta}{d}\Prth{\int_{\xi=0}^{\infty}v_{0}\prth{\xi}\md \xi}}_{=: W\prth{0}}e^{-\frac{c\beta}{d}y}
		\bigg\Vert _{L^{1}_{y}\prth{\mathbb{R}_{+}^{}}}
		\leq
		\ell e^{-k \tau},
		\qquad
		\forall \tau>0.
	\end{equation*}
	As a consequence, to establish \eqref{EQ_asymptotic_behavior_solutions_larger_half}, we only need to bound
	$
	\Vertii{
		{w}\prth{\tau,y} - \tilde{w}\prth{\tau,y}
		}{L^{1}_{y}\prth{\mathbb{R}_{+}^{}}}
	$
	which can be rewritten with the Duhamel's principle \cite{GigaNonlinear10} as
\begin{align}
	\Vertii{
	{w}\prth{\tau,y} - \tilde{w}\prth{\tau,y}
	}{L^{1}_{y}\prth{\mathbb{R}_{+}^{}}}
	& =
	\Vertii{
	\int_{s=0}^{\tau}
	{\Sb}_{s}\crocHH{-\eta\prth{\tau-s} \partial_{y}\prth{yw\prth{\tau-s,y}}}
	\md s
	}{L^{1}_{y}\prth{\mathbb{R}_{+}^{}}}\label{EQ_Duhamel_1}
\end{align}
	where $\prth{{\Sb}_{s}}_{s>0}$ is the $C_{0}$-semigroup on $L^{1}_{y}\prth{\mathbb{R}_{+}^{}}$ associated with the linear problem \eqref{EQ_for_w_beta_larger_1/2_linear}.
	Expanding the expression of this semigroup with the fundamental solution $H$ \eqref{EQ_heat_kernel_linear} yields
	\begin{equation}\label{EQ_Duhamel_2}
	\Vertii{
	{w}\prth{\tau,y} - \tilde{w}\prth{\tau,y}
	}{L^{1}_{y}} \leq
	\int_{s=0}^{\tau}
		\vertII{\eta\prth{\tau-s}}
\underbrace{		\int_{y=0}^{\infty}
			\vertIII{
			\overbrace{
				\int_{\xi=0}^{\infty}
					-H\prth{s,y,\xi}
					\partial_{\xi}\croch{\xi w\prth{\tau-s,\xi}}
				\md \xi}^{=: I\prth{\tau,s,y}
				}}
		\,\md y}_{=:J\prth{\tau,s}}
	\,\md s.
	\end{equation}
	Integrating $I$ by parts gives
	$$
	I\prth{\tau,s,y} =
	\int_{\xi=0}^{\infty}
		\partial_{\xi}H\prth{s,y,\xi}
		\croch{\xi w\prth{\tau-s,\xi}}
	\,\md \xi,
	$$
	where $\partial_{\xi}H$ can be computed from the expression \eqref{EQ_heat_kernel_linear} of $H$ (with $c$ replaced by $c\beta$) :
	$$
	\partial_{\xi}H\prth{s,y,\xi} =
	\frac{1}{2ds\sqrt{4\pi ds}}
	\Croch{
		\prth{y+c\beta s-\xi}e^{-\frac{\prth{y+c\beta s-\xi}^{2}}{4ds}} -
		\prth{y+c\beta s+\xi}e^{-\frac{\prth{y+c\beta s+\xi}^{2}}{4ds}}e^{\frac{c\beta }{d}\xi}
	}.
	$$
	Hence, by Fubini's theorem,
	$$
\scalebox{0.98}{$\displaystyle 	J\prth{\tau,s} \leq
	\int_{\xi=0}^{\infty}
		\frac{\xi w\prth{\tau-s,\xi}}{2ds\sqrt{4\pi ds}}
\underbrace{			\int_{y=0}^{\infty}
				\Croch{
				\verti{y+c\beta s-\xi}e^{-\frac{\prth{y+c\beta s-\xi}^{2}}{4ds}} +
				\prth{y+c\beta s+\xi}e^{-\frac{\prth{y+c\beta s+\xi}^{2}}{4ds}}e^{\frac{c\beta }{d}\xi}
				}
			\md y}_{=:K\prth{\tau,s,\xi}}
	\md \xi, $}
	$$
	where $K$ can be bounded as follows
	\begin{align}
		K\prth{\tau,s,\xi}
		&= \int_{\omega=c\beta s-\xi}^{\infty}
			\verti{\omega}e^{-\frac{\omega^{2}}{4ds}}
		\,\md \omega
		+ e^{\frac{c\beta }{d}\xi}\int_{\omega=c\beta s+\xi}^{\infty}
			{\omega}e^{-\frac{\omega^{2}}{4ds}}
		\,\md \omega
		\nonumber\\[2mm]
		&\leq \int_{\mathbb{R}}^{}
			\verti{\omega}e^{-\frac{\omega^{2}}{4ds}}
		\,\md \omega
		+ 2ds \, e^{\frac{c\beta }{d}\xi} e^{-\frac{\prth{c\beta s+\xi}^{2}}{4ds}}
		\nonumber\\[2mm]
		&= 4ds
		+ 2ds \, e^{-\frac{\prth{c\beta s-\xi}^{2}}{4ds}}
		\nonumber\\[2mm]
		&\leq  6ds.
		\nonumber
	\end{align}
	As a result
	\begin{equation}\label{EQ_control_J_beta_larger_1_2}
		J\prth{\tau,s} \leq \frac{3}{\sqrt{4\pi ds}}
		\int_{\xi=0}^{\infty}
			\xi w\prth{\tau-s, \xi}
		\,\md \xi.
	\end{equation}
	To proceed with our analysis, we need to establish a uniform upper bound on the first moment of $w$:
	\begin{lemma}[Upper bound on the first moment of $w$]\label{LE_upper_bound_first_moment_beta_1_2}
		Assume that $\beta\in(\frac{1}{2},{1}]$ and that $v_{0}$ is bounded, compactly supported, and nonnegative in $\mathbb{R}_{+}$.
		Then there exists $\Lambda>0$, depending on $c$, $d$, $\beta$, $\vertii{v_{0}}{L^{\infty}}$ and $\max\prth{\text{supp}\prth{v_{0}}}$, such that
		$$
		\int_{\xi=0}^{\infty}
			\xi w\prth{\tau, \xi}
		\,\md \xi
		\leq \Lambda,
		\qquad
		\forall \tau > 0.
		$$
	\end{lemma}

\begin{proof}[\texorpdfstring{Proof of \hyperref[LE_upper_bound_first_moment_beta_1_2]{Lemma \ref*{LE_upper_bound_first_moment_beta_1_2}}}{Proof of Lemma \ref*{LE_upper_bound_first_moment_beta_1_2}}]
		The proof follows from a comparison argument. Define, for $\lambda>0$,
		\begin{equation}\label{EQ_super_solution_beta_greater_1/2}
		\begin{array}{llll}
			\overline{w}\prth{\tau,y}:=
			\lambda		
			\exp\Prth{
				-\dfrac{c\beta }{d}y
				+\dfrac{\eta\prth{\tau}}{2d}y^{2}
			},
			& \quad &
			\tau>0, & y>0.
		\end{array}
		\end{equation}
		Then we can check that, for any $\tau>0$ and $y>0$,
		$$
		\Lb\overline{w} : = \partial_{t}\overline{w} - d \partial_{yy}\overline{w}
		- \partial_{y} \crocHH{ \prth{c\beta -\eta\prth{\tau}y}\overline{w}}
		= \frac{y^{2}\eta'\prth{t}}{2d}\overline{w},
		$$
		and for any $\tau>0$,
		$$
		-d\partial_{y}\overline{w}\prth{\tau,0} - c\beta  \overline{w}\prth{\tau,0} = 0,
		\qquad \text{at $y=0$}.
		$$
		In particular,
		$\text{sign}\prth{\Lb\overline{w}} = \text{sign}\prth{\eta'\prth{\tau}} = \text{sign}\prth{1-\beta}$, so that $\overline{w}$ is a super solution to problem \eqref{EQ_for_w_beta_larger_1/2_linear} since $\frac{1}{2}<\beta\leq 1$.
		As a result, by taking $\lambda>0$ large enough\footnote{This requires that $\lambda$ depends on $c$, $d$, $\beta$, $\vertii{v_{0}}{L^{\infty}}$ and $\max\prth{\text{supp}\prth{v_{0}}}$.} so that
		$
		\overline{w}|_{\tau=0}>w|_{\tau=0}
		$
		for any $y>0$, we have $\overline{w}>w$ for any $\tau>0$ and $y>0$.

		Considering that $\eta$ is nonpositive, we can finally write
		$$
		\int_{\xi=0}^{\infty}
			\xi \, w\prth{\tau, \xi}
		\,\md \xi \leq 
		\int_{\xi=0}^{\infty}
			\xi \, \overline{w}\prth{\tau, \xi}
		\,\md \xi \leq 
		\lambda \int_{\xi=0}^{\infty}
			\xi \, e^{-\frac{c\beta }{d}\xi}
		\,\md \xi  =: \Lambda.
		$$
	\end{proof}
	
	We now turn back to \eqref{EQ_Duhamel_2}, and obtain from \Cref{LE_upper_bound_first_moment_beta_1_2} and \eqref{EQ_control_J_beta_larger_1_2} that
	\begin{align}
		\Vertii{
		{w}\prth{\tau,y} - \tilde{w}\prth{\tau,y}
		}{L^{1}_{y}\prth{\mathbb{R}_{+}^{}}}
		&\leq
		\int_{s=0}^{\tau}
			\vertI{\eta\prth{\tau-s}}
			\frac{3\Lambda}{\sqrt{4\pi ds}}
		\,\md s \nonumber\\[2mm]
		&=
		\frac{3\Lambda\prth{1-\beta}}{\prth{2\beta-1}\sqrt{4\pi d}}
		\underbrace{\int_{s=0}^{\tau}
			\frac{1}{\frac{1}{2\beta-1}+\tau-s} \,
			\frac{1}{\sqrt{s}}
		\,\md s}_{= : L\prth{\tau}}. \nonumber
	\end{align}
	It remains to control the convolution $L\prth{\tau}$. To achieve this, let us set $s=\tau\sigma$ in the integral in $L\prth{\tau}$. This yields
	\begin{align}
		L\prth{\tau}
		&=
		\int_{\sigma=0}^{1}
			\frac{\tau}{\frac{1}{2\beta-1}+\tau-\tau\sigma} \,
			\frac{1}{\sqrt{\tau\sigma}}
		\,\md \sigma. \nonumber\\[2mm]
		&=
		\int_{\sigma=0}^{1/2}
			\frac{\tau}{\frac{1}{2\beta-1}+\tau-\tau\sigma} \,
			\frac{1}{\sqrt{\tau\sigma}}
		\,\md \sigma
		+
		\int_{\sigma=1/2}^{1}
			\frac{\tau}{\frac{1}{2\beta-1}+\tau-\tau\sigma} \,
			\frac{1}{\sqrt{\tau\sigma}}
		\,\md \sigma. \nonumber\\[2mm]
		&\leq 
		\frac{2}{\sqrt{\tau}}
		\int_{\sigma=0}^{1/2}
			\frac{\md \sigma}{\sqrt{\sigma}}
		+
		\frac{\sqrt{2}}{\sqrt{\tau}}
		\int_{\sigma=1/2}^{1}
			\frac{\tau}{\frac{1}{2\beta-1}+\tau-\tau\sigma}
		\,\md \sigma. \nonumber\\[2mm]
		&=
		\frac{2\sqrt{2}}{\sqrt{\tau}}
		+
		\frac{\sqrt{2}\log\Prth{1+\frac{2\beta-1}{2}\,\tau}}{\sqrt{\tau}}. \nonumber
	\end{align}
	Now using the relation \eqref{EQ_rel_between_tau_and_t} between the variables $\tau$ and $t$, namely
	\begin{equation*}
		\tau = 
		\frac{1}{2\beta-1} \crocHH{\prth{1+t}^{2\beta-1} - 1}
		\quad \iff \quad
		t = 
		\prtHH{1+\prth{2\beta-1}\tau}^{\frac{1}{2\beta-1}} - 1,
	\end{equation*}
	it can be shown that, for any $t>1$,
	$$
	\frac{1}{\sqrt{\tau}} \leq \frac{\sqrt{\prth{2\beta-1}/\prth{1-2^{1-2\beta}}}}{\prth{1+t}^{\beta-\frac{1}{2}}}
	\qquad
	\text{and}
	\qquad
	\log\Prth{1+\frac{2\beta-1}{2}\,\tau} \leq 
	\prth{2\beta-1}\log\prth{1+t}.
	$$
	As a result, there exists $k>0$, depending on $c$, $d$, $\beta$, $\vertii{v_{0}}{L^{\infty}}$ and $\max\prth{\text{supp}\prth{v_{0}}}$, such that
	\begin{equation}\label{EQ_almost_there_beta_ge_1_2}
		\Vertii{
			{w}\prth{\tau,y} - \tilde{w}\prth{\tau,y}
		}{L^{1}_{y}\prth{\mathbb{R}_{+}^{}}}
		\leq \frac{k \log\prth{1+t}}{\prth{1+t}^{\beta-\frac{1}{2}}},
		\qquad
		\forall t>1.
	\end{equation}
	Combining \eqref{EQ_almost_there_beta_ge_1_2} and the exponential convergence of $\tilde{w}\prth{\tau,\point}$ toward $W(y)=W\prth{0} e^{-\frac{c\beta}{d}y}$ (see \hyperref[TH_fundamental_sol_for_beta_1]{Theorem \ref*{TH_fundamental_sol_for_beta_1} \emph{(ii)}} with $c$ is replaced by $c\beta$),
	we reach, up to increasing $\ell$,
	\begin{equation}\label{EQ_almost_there_beta_ge_1_2_bis}
		\Vertii{
			{w}\prth{\tau,y} - W\prth{y}
		}{L^{1}_{y}\prth{\mathbb{R}_{+}^{}}}
		\leq \frac{\ell \log\prth{1+t}}{\prth{1+t}^{\beta-\frac{1}{2}}},
		\qquad
		\forall t>1.
	\end{equation}
	Now making the change of variable $y=\frac{b'\prth{t}}{c\beta}\,x$ in the left-hand-side of \eqref{EQ_almost_there_beta_ge_1_2_bis} and recalling that $v\prth{t,x} = \frac{b'\prth{t}}{c\beta}w\prth{\tau,y}$ (see \eqref{EQ_rescaling_super_sqrt}), we finally obtain
	$$
	\Vertii{
		{v}\prth{t,x} - \frac{b'\prth{t}}{c\beta} W\prtHHH{\frac{b'\prth{t}}{c\beta}\,x} 
	}{L^{1}_{x}\prth{\mathbb{R}_{+}^{}}}
	\leq \frac{\ell \log\prth{1+t}}{\prth{1+t}^{\frac{1}{2}}},
	\qquad
	\forall t>1,
	$$
	which gives the control \eqref{EQ_asymptotic_behavior_solutions_larger_half} announced in \hyperref[TH_asymptotic_behavior_solutions]{Theorem~\ref*{TH_asymptotic_behavior_solutions}~\emph{(iii)}} and concludes the proof.
\end{proof}

\paragraph*{Acknowledgement}
S.T. acknowledges support from the European Research Council (ERC) under the European Union’s Horizon 2020 research and innovation program (grant agreement No 865711). M.Z. acknowledges support from the ANR via the project ReaCh ANR-23-CE40-0023-01, and from the Occitanie region, the European Regional Development Fund (ERDF), and the French government, through the France 2030 project managed by the National Research Agency (ANR) ANR-22-EXES-0015.

The authors would like to thank everyone with whom they had the opportunity to discuss this paper, in particular
\href{http://math.univ-lyon1.fr/homes-www/lepoutre/}{Thomas Lepoutre},
whose insightful advice has marked significant milestones in advancing this work.


\bibliographystyle{siam}
\bibliography{biblio}

\end{document}